\documentclass[12pt]{amsart}

\usepackage{amsmath}
\usepackage{amsfonts}
\usepackage{amssymb}
\usepackage[all]{xy}           
\usepackage{bbding}
\usepackage{txfonts}
\usepackage{amscd}

\usepackage[shortlabels]{enumitem}
\usepackage{ifpdf}
\ifpdf
  \usepackage[colorlinks,final,backref=page,hyperindex]{hyperref}
\else
  \usepackage[colorlinks,final,backref=page,hyperindex]{hyperref}
\fi
\usepackage{tikz}
\usepackage[active]{srcltx}

\makeatletter
\numberwithin{equation}{section}

\newtheorem{thm}{Theorem}[section]
\newtheorem{lem}[thm]{Lemma}

\newtheorem{pro}[thm]{Proposition}
\newtheorem{ex}[thm]{Example}
\newtheorem{rmk}[thm]{Remark}
\newtheorem{defi}[thm]{Definition}

\newcommand{\fl}{\mathbf l}
\newcommand{\fr}{\mathbf r}
\newcommand{\ffl}{\tilde{\mathbf{l}}}
\newcommand{\ffr}{\tilde{\mathbf{r}}}

\newcommand{\g}{\mathfrak{g}}

\newcommand{\gl}{\mathfrak{gl}}

\newcommand{\kl}{\mathfrak{l}}
\newcommand{\kr}{\mathfrak{r}}
\newcommand{\kkl}{\tilde{\mathfrak{l}}}
\newcommand{\kkr}{\tilde{\mathfrak{r}}}

\newcommand{\bz}{\mathbb{Z}}

\newcommand{\ad}{\mathrm{ad}}
\newcommand{\PYBE}{\mathrm{PYBE}}
\newcommand{\CYBE}{\mathrm{CYBE}}
\newcommand{\LYBE}{\mathrm{LYBE}}

\newcommand{\id}{\mathrm{id}}
\newcommand{\Img}{\mathrm{Im}}

\begin{document}

\title[Leibniz bialgebra constructed from Lie bialgebras and perm bialgebras ]
{Leibniz bialgebras constructed by tensor product from Lie bialgebras and perm bialgebras}

\author{Bo Hou}
\address{School of Mathematics and Statistics, Henan University, Kaifeng 475004, China}
\email{bohou1981@163.com}

\author{Ru Li}
\address{School of Mathematics and Statistics, Henan University, Kaifeng 475004,
China}
\email{13037698973@163.com}


\begin{abstract}
The construction problem of Leibniz bialgebras from Lie bialgebras and perm bialgebras
is considered in this paper. We show that there is a Leibniz algebra structure on the
tensor product of a Lie algebra and a perm algebra, and elevate this conclusion to
the level of bialgebra. We prove that the tensor product of a quadratic Lie algebra and a
perm bialgebra has a Leibniz bialgebra structure, and this Leibniz bialgebra structure
is coboundary (resp. quasi-triangular, triangular, factorizable) if the original perm
bialgebra is coboundary (resp. quasi-triangular, triangular, factorizable). Moreover,
we constructed an infinite-dimensional Leibniz bialgebra using the tensor
product of a finite-dimensional Lie bialgebra and a quadratic $\bz$-graded perm algebra.
Quasi-triangular and triangular infinite-dimensional Leibniz bialgebras are considered.
 \end{abstract}

\dedicatory{School of Mathematics and Statistics, Henan University, Kaifeng 475004, China\\
Email: bohou1981@163.com, 13037698973@163.com}

\keywords{Leibniz bialgebra, Lie bialgebra, perm bialgebra,
Yang-Baxter equation, $\mathcal{O}$-operator}
\makeatletter
\@namedef{subjclassname@2020}{\textup{2020} Mathematics Subject Classification}
\makeatother
\subjclass[2020]{
17A32, 
17A60, 
17B38, 
17B62. 
}

\maketitle

\vspace{-6mm}
\tableofcontents 



\vspace{-8mm}

\section{Introduction}\label{sec:intr}

Leibniz algebras are certain generalizations of Lie algebras, which lack the
anti-symmetry property of Lie algebras. Leibniz algebras were first studied by
Bloh \cite{Blo} in 1965 and later introduced by Loday \cite{Lod, LP} with the motivation
in the study of the periodicity in algebraic $K$-theory. In recent years, Leibniz algebras
were studied from different aspects due to applications in both mathematics and physics:
integration of Leibniz algebras were studied in \cite{Cov} and deformation quantization of
Leibniz algebras was studied in \cite{DW}.

A bialgebra structure on a given algebra structure is obtained by a comultiplication
together with some compatibility conditions between the multiplication and the
comultiplication. A famous example of bialgebra structures is Lie bialgebras \cite{Dri},
which can be viewed as a linearization of Poisson-Lie groups.
In recent years, the bialgebra theories of various algebra structures have been
extensively developed, such as left-symmetric bialgebras (also called pre-Lie
bialgebras)\cite{Bai}, Jordan bialgebras \cite{Zhe}, Novikov bialgebras \cite{HBG},
perm bialgebras \cite{Hou,LZB}, Jacobi-Jordan bialgebras \cite{BCHM}, antisymmetric
infinitesimal bialgebras \cite{Bai1}, etc. Recently, quasi-triangular bialgebra
structure, especially factorizable bialgebra structure, have received a lot of attention.
Factorizable Lie bialgebras \cite{LS}, factorizable antisymmetric infinitesimal bialgebras
\cite{SW}, factorizable pre-Lie bialgebras \cite{WBLS} and factorizable Novikov
bialgebras \cite{CH} have been further studied. Leibniz algebras are
certain generalizations of Lie algebras. Further attention has been paid to
Leibniz bialgebras, especially some special Leibniz bialgebras such as
triangular Leibniz bialgebras and factorizable Leibniz bialgebras \cite{TS,BLST}.
In this paper, we further consider the construction problems of triangular Leibniz
bialgebras and factorizable Leibniz bialgebras.

Recently, the relationship between different types of bialgebra structures
has received a lot of attention. Since the operad of (left) Novikov algebras
and the operad of right Novikov algebras are Koszul dual, in \cite{HBG}, Hong, Bai
and Guo have proposed a method for constructing infinite-dimensional Lie bialgebras
using the affinization of Novikov bialgebras. Similarly, since the operads of perm
algebras and pre-Lie algebras are the Koszul dual each other, Lin, Zhou and Bai have
constructed infinite-dimensional Lie bialgebras by using the pre-Lie bialgebras
and perm bialgebras, respectively \cite{LZB}. Some relationships between differential
antisymmetric infinitesimal bialgebras, Novikov bialgebras and Lie bialgebras
were considered in \cite{HBG2,CH}. In \cite{GH}, we have provided some methods for
constructing Lie bialgebras by discussing in detail the relationship between
dendriform $D$-bialgebras, pre-Lie bialgebras, antisymmetric infinitesimal bialgebras
and Lie bialgebras. In \cite{Hou1}, we have provided a method for constructing
infinite-dimensional Poisson bialgebras by the affinization of pre-Lie bialgebras
and Zinbiel bialgebras.
Moreover, the relationship between Gel'fand-Dorfman bialgebras and Lie conformal
bialgebras have been discussed in \cite{HBG1}. In this paper, we mainly consider
the construction problem of Leibniz bialgebras from Lie bialgebras and perm bialgebras.

Note that there is a Leibniz algebra structure on the tensor product of a Lie
algebra and a perm algebra. We first consider the coalgebra versions of the conclusion,
and obtain a Leibniz coalgebra structure on the tensor product of a Lie coalgebra and
a perm coalgebra. Since the quadratic Lie algebra can naturally induce the Lie coalgebra
structure on its dual space, we show that there is a Leibniz bialgebra structure
on the tensor product of a quadratic Lie algebra and a perm bialgebra.
The quasi-triangular Leibniz bialgebras are related to the solution of the Leibniz
Yang-Baxter equation ($\LYBE$). To construct a quasi-triangular Leibniz
bialgebra, we need to start by discussing the relationship between the solutions of
$\LYBE$ and the solutions of classical Yang-Baxter equation ($\CYBE$) and the
solutions of perm Yang-Baxter equation ($\PYBE$). The $\CYBE$ arose from the
study of inverse scattering theory in the 1980s and was recognized as the semi-classical
limit of the quantum Yang-Baxter equation \cite{Yang,Bax}. The $\PYBE$ was introduced
in \cite{Hou,LZB,Lin} to study some special perm bialgebras. In fact, a symmetric
solutions of the $\PYBE$ (resp. $\LYBE$) induces a triangular perm bialgebra
(resp. triangular Leibniz bialgebra). Here, from solutions of the $\PYBE$ in a perm
algebra, we first obtain constructions of symmetric solutions of the $\LYBE$ in the induced
Leibniz algebra which is obtained from the tensor product of the perm algebra and a
quadratic Lie algebra. And this construction maintains the symmetry of the solution
and the invariance of the symmetric parts the solution. Therefore, we obtain the
following theorem.

\smallskip\noindent
{\bf Theorem I} (Theorems \ref{thm:permbia-Lbi} and \ref{thm:indu-sLbia})
{\it Let $(B, \diamond, \nu)$ be a perm bialgebra and $(\g, [-,-], \omega)$ be a quadratic
Lie algebra and $(\g\otimes B, \ast)$ be the induced Leibniz algebra by
$(\g, [-,-])$ and $(B, \diamond)$, i.e., $(g_{1}, b_{1})\ast(g_{2}, b_{2})
=[g_{1}, g_{2}]\otimes(b_{1}\diamond b_{2})$ for any $g_{1}, g_{2}\in\g$ and
$b_{1}, b_{2}\in B$. Define a linear map $\vartheta: \g\otimes B\rightarrow
(\g\otimes B)\otimes(\g\otimes B)$ by
$$
\vartheta(g\otimes b)=\delta_{\omega}(g)\bullet\nu(b))
:=\sum_{(g)}\sum_{(b)}(g_{(1)}\otimes b_{(1)})\otimes(g_{(2)}\otimes b_{(2)}),
$$
for any $g\in\g$ and $b\in B$, where $\delta_{\omega}(g)=\sum_{(g)}g_{(1)}\otimes g_{(2)}$
and $\nu(b)=\sum_{(b)}b_{(1)}\otimes b_{(2)}$ in the Sweedler notation.
Then $(\g\otimes B, \ast, \vartheta)$ is a Leibniz bialgebra. Moreover,
we get that $(\g\otimes B, \ast, \vartheta)$ is a coboundary (resp.
quasi-triangular, triangular, factorizable) completed Leibniz bialgebra if
$(B, \diamond, \nu)$ is coboundary (resp. quasi-triangular,
triangular, factorizable).}

\smallskip
In \cite{Kup}, Kupershmidt found that the $\CYBE$ in tensor form on Lie algebras
can be converted into an $\mathcal{O}$-operator associated to the coadjoint representation.
This conclusion has been proven to be valid for various algebra structures.
Nowadays, the $\mathcal{O}$-operator is regarded as an operator form of a solution
of the classical Yang-Baxter equation. On the basis of proving that a symmetric solution
of the $\PYBE$ induces a solution of the $\LYBE$ in the induced Leibniz algebra,
we obtain a method for constructing $\mathcal{O}$-operators of the induced Leibniz algebra
from $\mathcal{O}$-operators of the original perm algebra. More precisely,
we have the following commutative diagram:
$$
\xymatrix@C=3cm@R=0.5cm{
\txt{$r$ \\ {\tiny a symmetric solution} \\ {\tiny of the $\PYBE$ in $(B, \diamond)$}}
\ar[d]_-{{\rm Pro.}~\ref{pro:PYBE-LYBE}}\ar[r]^-{{\rm Pro.}~\ref{pro:o-perm}} &
\txt{$r^{\sharp}$\\ {\tiny an $\mathcal{O}$-operator of $(B, \diamond)$} \\
{\tiny associated to $(B^{\ast}, \ffl_{B}^{\ast}, \ffl_{B}^{\ast}-\ffr_{B}^{\ast})$}}
\ar[d]^-{\mbox{$-\otimes\kappa^{\sharp}$}} \\
\txt{$\widetilde{r}$ \\ {\tiny a symmetric solution} \\ {\tiny of the $\LYBE$ in
$(\g\otimes B, \ast)$}} \ar[r]^-{{\rm Pro.}~\ref{pro:o-leib}}
& \txt{$\widetilde{r}^{\sharp}=r^{\sharp}\otimes\kappa^{\sharp}$ \\
{\tiny an $\mathcal{O}$-operator of $(\g\otimes B, \ast)$ } \\
{\tiny associated to $((\g\otimes B)^{\ast}, \fl_{\g\otimes B}^{\ast},
-\fl_{\g\otimes B}^{\ast}-\fr_{\g\otimes B}^{\ast})$}}}
$$

Note that the roles of perm algebras and Lie algebras are symmetric, there is also a
Leibniz bialgebra structure on the tensor product of a Lie bialgebra and a quadratic
perm algebra. We apply this method to the construction of infinite-dimensional Leibniz
bialgebras. We introduce a completed tensor product such that there is a $\bz$-graded
Leibniz coalgebra structure on the tensor product of a Lie coalgerba and a
$\bz$-graded perm coalgebra. By elevating this conclusion to the level of bialgebras,
the following conclusion can be obtained.

\smallskip\noindent
{\bf Theorem II} (Theorems \ref{thm:Lie+perm=L} and \ref{thm:indu-triLib})
{\it Let $(\g, [-,-], \delta)$ be a finite-dimensional Lie bialgebra,
$(B=\oplus_{i\in\bz}B_{i}, \diamond, \varpi)$ be a quadratic $\bz$-graded perm algebra
and $(\g\otimes B, \ast)$ be the induced $\bz$-graded Leibniz algebra from $(\g, [-,-])$
by $(B, \diamond)$. Define a linear map $\vartheta: \g\otimes B\rightarrow(\g\otimes B)
\otimes(\g\otimes B)$ by
$$
\vartheta(g\otimes b)=\delta(g)\bullet\nu_{\varpi}(b)
:=\sum_{(g)}\sum_{i,j,\alpha}(g_{(1)}\otimes b_{1,i,\alpha})
\otimes(g_{(2)}\otimes b_{2,j,\alpha}),
$$
for any $g\in\g$ and $b\in B$, where $\delta(g)=\sum_{(g)}g_{(1)}\otimes g_{(2)}$
in the Sweedler notation and $\nu_{\varpi}(b)=\sum_{i,j,\alpha}
b_{1,i,\alpha}\otimes b_{2,j,\alpha}$. Then $(\g\otimes B, \ast, \vartheta)$ is a
completed Leibniz bialgebra.

In particular, we get that $(\g\otimes B, \ast, \vartheta)$ is a coboundary (resp.
quasi-triangular, triangular) completed Leibniz bialgebra if
$(\g, [-,-], \delta)$ is coboundary (resp. quasi-triangular, triangular).}

\smallskip
The paper is organized as follows.
In Section \ref{sec:Leib-alg}, we recall the notions of Leibniz algebras, representations
of Leibniz algebras and Leibniz bialgebras. In particular, we show that there is a
Leibniz algebra structure on the tensor product of a Lie algebra and a perm algebra.
In Section \ref{sec:qtbia-perm}, we show in Theorem \ref{thm:permbia-Lbi} that there is
a natural Leibniz bialgebra structure on the tensor product of a perm bialgebra and a
quadratic Lie algebra. Constructions of (symmetric) solutions of the $\LYBE$ in the
induced Leibniz bialgebra from the (symmetric) solutions of the $\PYBE$ in a
perm algebra are obtained in Proposition \ref{pro:PYBE-LYBE}. Therefore, we get
the induced Leibniz algebra is coboundary (resp. quasi-triangular, triangular, factorizable)
if the original perm bialgebra is coboundary (resp. quasi-triangular, triangular,
factorizable). In Section \ref{sec:infi-bia}, we show that there is
an infinite-dimensional Leibniz bialgebra structure on the tensor product
of a quadratic $\bz$-graded perm algebra and a (finite-dimensional) Lie bialgebra
(see Theorem \ref{thm:Lie+perm=L}). By introducing the completed solution of the
$\LYBE$, we discuss the relationship between the completed solution of the $\LYBE$
in the induced infinite-dimensional Leibniz algebra and the solution of the $\CYBE$
in the original Lie algebra. We give in Theorem \ref{thm:indu-triLib} that the induced
infinite-dimensional Leibniz bialgebra is coboundary (resp. quasi-triangular, triangular)
if the original (finite-dimensional) Lie bialgebra is coboundary (resp.
quasi-triangular, triangular).

Throughout this paper, we fix $\Bbbk$ as a field of characteristic zero.
All the vector spaces, algebras are over $\Bbbk$ and are finite-dimensional
unless otherwise specified, and all tensor products are also over $\Bbbk$.
We denote the identity map by $\id$. For any finite-dimensional $\Bbbk$-vector
space $V$, we denote $V^{\ast}$ the dual space.

\bigskip

\section{Leibniz algebras and Leibniz bialgebras} \label{sec:Leib-alg}
In this section, we recall the notions of Leibniz algebras, representations of Leibniz
algebras and Leibniz bialgebras.

\begin{defi}\label{def:Le-alg}
A (left) {\bf Leibniz algebra} $(A, \ast)$ is a vector space $A$ together with a bilinear map
$\ast: A\otimes A\rightarrow A$ satisfying the following Leibniz identity:
$$
a_{1}\ast(a_{2}\ast a_{3})=(a_{1}\ast a_{2})\ast a_{3}+a_{2}\ast(a_{1}\ast a_{3}),
$$
for any $a_{1}, a_{2}, a_{3}\in A$.
\end{defi}

Let $(A, \ast)$ and $(A', \ast')$ be two Leibniz algebras. A linear
map $f: A\rightarrow A'$ is called a {\bf homomorphism of Leibniz algebras}
if $f(a_{1}\ast a_{2})=f(a_{1})\ast'f(a_{2})$ for any $a_{1}, a_{2}\in A$.
A Leibniz algebra $(A, \ast)$ is called a {\bf Lie algebra} if the product $\ast$ is
anticommutative, i.e., $a_{1}\ast a_{2}=-a_{2}\ast a_{1}$ for any $a_{1}, a_{2}\in A$.
In this case, we usually denote the product $\ast$ by the Lie Bracket $[-,-]$.
Recall that a {\bf perm algebra} is a pair $(B, \diamond)$, where $B$ is a vector space
and $\diamond: B\otimes B\rightarrow B$ is a bilinear operator such that for any
$b_{1}, b_{2}, b_{3}\in B$,
$$
b_{1}\diamond(b_{2}\diamond b_{3})=(b_{1}\diamond b_{2})\diamond b_{3}
=(b_{2}\diamond b_{1})\diamond b_{3}.
$$
In a perm algebra $(B, \diamond)$, we also have $b_{1}\diamond(b_{2}\diamond b_{3})
=b_{2}\diamond(b_{1}\diamond b_{3})$. We can get a Leibniz algebra by the tensor
product of a Lie algebra and a perm algebra.

\begin{pro}\label{pro:Lie+perm}
Let $(\g, [-,-])$ be a Lie algebra and $(B, \diamond)$ be a perm algebra.
Define a bilinear map $\ast: (\g\otimes B)\otimes(\g\otimes B)\rightarrow \g\otimes B$ by
$$
(g_{1}\otimes b_{1})\ast(g_{2}\otimes b_{2}):=[g_{1}, g_{2}]\otimes(b_{1}\diamond b_{2}),
$$
for any $g_{1}, g_{2}\in\g$ and $b_{1}, b_{2}\in B$.
Then $(\g\otimes B, \ast)$ is a Leibniz algebra.
\end{pro}

\begin{proof}
For any $g_{1}, g_{2}, g_{3}\in\g$ and $b_{1}, b_{2}, b_{3}\in B$, we have
\begin{align*}
&\;((g_{1}\otimes b_{1})\ast(g_{2}\otimes b_{2}))\ast(g_{3}\otimes b_{3})
+(g_{2}\otimes b_{2})\ast((g_{1}\otimes b_{1})\ast(g_{3}\otimes b_{3}))\\
=&\;[[g_{1}, g_{2}], g_{3}]\otimes(b_{1}\diamond b_{2})\diamond b_{3}
+[g_{2}, [g_{1}, g_{3}]]\otimes b_{2}\diamond(b_{1}\diamond b_{3})\\
=&\;[g_{1}, [g_{2}, g_{3}]]\otimes b_{1}\diamond(b_{2}\diamond b_{3})\\
=&\;(g_{1}\otimes b_{1})\ast((g_{2}\otimes b_{2})\ast(g_{3}\otimes b_{3})).
\end{align*}
Thus, $(\g\otimes B, \ast)$ is a Leibniz algebra.
\end{proof}

We now consider the representations of a Leibniz algebra.

\begin{defi}\label{def:Le-rep}
Let $(A, \ast)$ be a Leibniz algebra, $V$ be a vector space
and $\kl, \kr: A\rightarrow\gl(V)$ be two linear maps. Then $(V, \kl, \kr)$ is
called a {\bf representation of $(A, \ast)$} if for any $a_{1}, a_{2}\in A$,
\begin{align*}
&\qquad\; \kl(a_{1}\ast a_{2})=\kl(a_{1})\kl(a_{2})-\kl(a_{2})\kl(a_{1}),\\
&\kr(a_{1})\kr(a_{2})=\kr(a_{2}\ast a_{1})-\kl(a_{2})\kr(a_{1})=-\kr(a_{1})\kl(a_{2}).
\end{align*}
\end{defi}

For the representations of Leibniz algebras, we have the following equivalent
characterization.

\begin{pro}\label{pro:rep-lei}
Let $(A, \ast)$ be a Leibniz algebra, $V$ be a vector space and $\kl, \kr:
A\rightarrow\gl(V)$ be two linear maps. Define a bilinear map $\tilde{\ast}: (A\oplus V)
\otimes(A\oplus V)\rightarrow A\oplus V$ by
$$
(a_{1}, v_{1})\tilde{\ast}(a_{2}, v_{2}):=\big(a_{1}\ast a_{2},\ \
\kl(a_{1})(v_{2})+\kr(a_{2})(v_{1})\big),
$$
for any $a_{1}, a_{2}\in A$ and $v_{1}, v_{2}\in V$. Then $(V, \kl, \kr)$
is a representation of $(A, \ast)$ if and only if $(A\oplus V, \tilde{\ast})$ is a
Leibniz algebra. We denote this Leibniz algebra by $A\ltimes V$, and call it the
{\bf semidirect product} of $(A, \ast)$ by the representation $(V, \kl, \kr)$.
\end{pro}

\begin{proof}
It is a straightforward check.
\end{proof}

Let $(V, \kl, \kr)$ and $(V', \kl', \kr')$ be two representations
of a Leibniz algebra $(A, \ast)$. A linear map $f: V\rightarrow V'$ is called
a {\bf representation homomorphism} if $f(\kl(a)(v))=\kl'(a)(f(v))$ and
$f(\kr(a)(v))=\kr'(a)(f(v))$ for any $a\in A$, $v\in V$. A representation homomorphism
$f$ is said to be an {\bf isomorphism} if $f$ is a bijection. Define the left
multiplication map $\fl_{A}: A\rightarrow\gl(A)$ and the right multiplication map
$\fr_{A}: A\rightarrow\gl(A)$ by $\fl_{A}(a_{1})(a_{2})=a_{1}\ast a_{2}$,
and $\fr_{A}(a_{1})(a_{2})=a_{2}\ast a_{1}$ respectively for all $a_{1}, a_{2}\in A$.
Then $(A, \fl_{A}, \fr_{A})$ is a representation of $(A, \ast)$,
which is called the {\bf regular representation} of $(A, \ast)$.

Let $V$, $W$ be two finite-dimensional $\Bbbk$-vector spaces. We denote
$\langle-,-\rangle$ the natural pairing between the spaces $V$ and $V^{\ast}$,
i.e., $\langle\xi, v\rangle:=\xi(v)\in\Bbbk$ for any $v\in V$ and $\xi\in V^{\ast}$.
For a linear map $f: V\rightarrow W$, we define the map $f^{\ast}: W^{\ast}
\rightarrow V^{\ast}$ by $\langle f^{\ast}(\xi), v\rangle
=\langle\xi, f(v)\rangle$ for any $v\in V$ and $\xi\in W^{\ast}$.
Let $(A, \ast)$ be a finite-dimensional Leibniz algebra and $(V, \kl, \kr)$ be
a representation of it. We define linear maps $\kl^{\ast}, \kr^{\ast}:
A\rightarrow\gl(V^{\ast})$ by
$$
\langle\kl^{\ast}(a_{1})(\xi),\; a_{2}\rangle=-\langle\xi,\; \kl(a_{1})(a_{2})\rangle,\qquad
\langle\kr^{\ast}(a_{1})(\xi),\; a_{2}\rangle=-\langle\xi,\; \kr(a_{1})(a_{2})\rangle,
$$
for any $a_{1}, a_{2}\in A$ and $\xi\in V^{\ast}$.
For a Leibniz algebra $(A, \ast)$, one can check that $(A^{\ast},
\fl_{A}^{\ast}, -\fl_{A}^{\ast}-\fr_{A}^{\ast})$ is a representation of $(A, \ast)$,
which is called the {\bf coregular representation} of $(A, \ast)$.

Let $\varpi(-, -)$ be a bilinear form on a Leibniz algebra $(A, \ast)$. Recall that
\begin{itemize}
\item[-] $\varpi(-, -)$ is called {\bf nondegenerate} if
     $\varpi(a_{1}, a_{2})=0$ for any $a_{2}\in A$, then $a_{1}=0$;
\item[-] $\varpi(-, -)$ is called {\bf skew-symmetric} if $\varpi(a_{1}, a_{2})
     =-\varpi(a_{2}, a_{1})$, for any $a_{1}, a_{2}\in A$;
\item[-] $\varpi(-, -)$ is called {\bf invariant} if $\varpi(a_{1}\ast a_{2},\;
     a_{3})=\varpi(a_{1},\; a_{2}\ast a_{3}+a_{3}\ast a_{2})$ for any $a_{1}, a_{2},
     a_{3}\in A$.
\end{itemize}
A Leibniz algebra $(A, \ast)$ with a nondegenerate skew-symmetric invariant
bilinear form $\varpi(-, -)$ is called a {\bf quadratic Leibniz algebra} and
denoted by $(A, \ast, \varpi)$. It is easy to see that
$\varpi(a_{1}\ast a_{2},\; a_{3})=-\varpi(a_{2},\; a_{1}\ast a_{3})$ for any
$a_{1}, a_{2}, a_{3}\in A$ if $(A, \ast, \varpi)$ is quadratic.
Moreover, one can check that $(A, \fl_{A}, \fr_{A})$ and $(A^{\ast}, \fl_{A}^{\ast},
-\fl_{A}^{\ast}-\fr_{A}^{\ast})$ are isomorphic as representations of
the Leibniz algebra $(A, \ast)$ if there exists a
nondegenerate skew-symmetric invariant bilinear form on $(A, \ast)$.

Next, we consider the bialgebra theory of Leibniz algebras, for the details see
\cite{TS,BLST}. Recall that a {\bf Leibniz coalgebra} $(A, \vartheta)$ is a vector space
$A$ with a linear map $\vartheta: A\rightarrow A\otimes A$ such that
$(\id\otimes\vartheta)\vartheta-(\tau\otimes\id)(\id\otimes\vartheta)\vartheta
=(\vartheta\otimes\id)\vartheta$, where $\tau: A\otimes A\rightarrow A\otimes A$ is
the twist map defined by $\tau(a_{1}\otimes a_{2}):=a_{2}\otimes a_{1}$ for all
$a_{1}, a_{2}\in A$.

\begin{defi}\label{def:Le-bialg}
A {\bf Leibniz bialgebra} is a triple $(A, \ast, \vartheta)$,
where $(A, \ast)$ is a Leibniz algebra, $(A, \vartheta)$ is a Leibniz coalgebra
and the following equations hold:
\begin{align*}
&\qquad\qquad\quad \tau((\fr_{A}(a_{2})\otimes\id)(\vartheta(a_{1})))
=(\fr_{A}(a_{1})\otimes\id)(\vartheta(a_{2})),\\
&\vartheta(a_{1}\ast a_{2})=(\id\otimes\fr_{A}(a_{2})-(\fl_{A}+\fr_{A})(a_{2})\otimes\id)
((\id\otimes\id+\tau)(\vartheta(a_{1})))\\[-1mm]
&\qquad\qquad\qquad\qquad\qquad\qquad\qquad +(\id\otimes\fl_{A}(a_{1})
+\fl_{A}(a_{1})\otimes\id)(\vartheta(a_{2})),
\end{align*}
for any $a_{1}, a_{2}\in A$.
\end{defi}

Let $(A, \ast)$ be a Leibniz algebra. We define a linear map
$F: A\rightarrow\gl(A\otimes A)$ by
$$
F(a)=(\fl_{A}+\fr_{A})(a)\otimes\id-\id\otimes\fr_{A}(a).
$$
An element $r\in A\otimes A$ is called {\bf symmetric} if $r=\tau(r)$;
called {\bf Leib-invariant} if $F(a)(r)=0$ for all $a\in A$.
If there exists an element $r\in A\otimes A$ such that $(A, \ast, \vartheta_{r})$ is a
Leibniz bialgebra, where $\vartheta_{r}: A\rightarrow A\otimes A$ is given by
\begin{align}
\vartheta_{r}(a)=F(a)(r),    \label{cobLeb}
\end{align}
for any $a\in A$, then $(A, \ast, \vartheta_{r})$ is called a {\bf coboundary
Leibniz bialgebra} associated with $r$. Let $(A, \ast)$ be a Leibniz algebra and
$r=\sum_{i}x_{i}\otimes y_{i}\in A\otimes A$. The equation
$$
\mathbf{L}_{r}=r_{12}\ast r_{13}-r_{12}\ast r_{23}-r_{23}\ast r_{12}+r_{23}\ast r_{13}=0
$$
is called the {\bf Leibniz Yang-Baxter equation} (or $\LYBE$) in the Leibniz
algebra $(A, \ast)$, where $r_{12}\ast r_{13}=\sum_{i,j}x_{i}x_{j}\otimes y_{i}
\otimes y_{j}$, $r_{12}\ast r_{23}=\sum_{i,j}x_{i}\otimes y_{i}x_{j}\otimes y_{j}$,
$r_{23}\ast r_{12}=\sum_{i,j}x_{j}\otimes x_{i}y_{j}\otimes y_{i}$ and
$r_{23}\ast r_{13}=\sum_{i,j}x_{j}\otimes x_{i}\otimes y_{i}y_{j}$.

\begin{pro}[\cite{BLST}]\label{pro:sLib-bia}
Let $(A, \ast)$ be a Leibniz algebra, $r\in A\otimes A$ and $\vartheta_{r}:
A\rightarrow A\otimes A$ be the linear map defined by Eq. \eqref{cobLeb}.
\begin{enumerate}
\item[$(i)$] If $r$ is a solution of the $\LYBE$ in $(A, \ast)$ and $r-\tau(r)$ is
     Leib-invariant, then $(A, \ast, \vartheta_{r})$ is a Leibniz bialgebra,
     which is called a {\bf quasi-triangular Leibniz bialgebra} associated with $r$.
\item[$(ii)$] If $r$ is a symmetric solution of the $\LYBE$ in $(A, \ast)$, then $(A, \ast,
     \vartheta_{r})$ is a Leibniz bialgebra, which is called a {\bf triangular Leibniz
     bialgebra} associated with $r$.
\end{enumerate}
\end{pro}

Let $(A, \ast)$ be a Leibniz algebra. For any $r\in A\otimes A$,
we define a linear map $r^{\sharp}: A^{\ast}\rightarrow A$ by
$$
\langle r^{\sharp}(\xi_{1}),\; \xi_{2}\rangle=\langle\xi_{1}\otimes\xi_{2},\; r\rangle,
$$
for any $\xi_{1}, \xi_{2}\in A^{\ast}$ and denote $\mathcal{I}=r^{\sharp}-\tau(r)^{\sharp}: A^{\ast}\rightarrow A$.

\begin{defi}[\cite{BLST}]\label{def:fact-Leib}
Let $(A, \ast)$ be a Leibniz algebra, $r\in A\otimes A$ and $(A, \ast, \vartheta_{r})$
be a quasi-triangular Leibniz bialgebra associated with $r$. If $\mathcal{I}=
r^{\sharp}-\tau(r)^{\sharp}: A^{\ast}\rightarrow A$ is an isomorphism of vector spaces,
then $(A, \ast, \vartheta_{r})$ is called a {\bf factorizable Leibniz bialgebra}.
\end{defi}

In a factorizable Leibniz bialgebra $(A, \ast, \vartheta_{r})$, each element $a\in A$
can be decomposed into $a=a_{+}+a_{-}$, where $a_{+}\in\Img(r^{\sharp})$ and
$a_{-}\in\Img(\tau(r)^{\sharp})$ \cite{BLST}. Note that $\mathcal{I}=0$ if
$(A, \ast, \vartheta_{r})$ is a triangular Leibniz bialgebra, we can view the
factorizable Leibniz bialgebra is the opposite of the triangular Leibniz bialgebra.
In this paper, we consider the finite-dimensional and infinite-dimensional
Leibniz bialgebras induced by Lie bialgebras and perm bialgebras.

\bigskip
\section{Quasi-triangular Leibniz bialgebras from perm bialgebras} \label{sec:qtbia-perm}
In this section, we recall the notions of perm bialgebras and quadratic Lie algebras.
We show that there is a Leibniz bialgebra on the tensor product of a perm bialgebra
and a quadratic Lie algebra, and prove that the induced Leibniz bialgebra is
quasi-triangular (resp. triangular, factorizable) if the perm bialgebra is
quasi-triangular (resp. triangular, factorizable).

For the construction of the Leibniz coalgebras, we can obtain the dual conclusion
for Proposition \ref{pro:Lie+perm}. Recall that  {\bf Lie coalgebra} $(\g, \delta)$
is a vector space $\g$ with a linear map $\delta: \g\rightarrow\g\otimes\g$,
such that $\tau\delta=-\delta$ and $(\id\otimes\delta)\delta
-(\tau\otimes\id)(\id\otimes\delta)\delta=(\delta\otimes\id)\delta$.
One can check that $(\g, \delta)$ is a Lie coalgebra if and only if
$(\g^{\ast}, \delta^{\ast})$ is a Lie algebra.

\begin{defi}\label{def:permcoalg}
A {\bf perm coalgebra} $(B, \nu)$ is a vector space $B$
with a linear map $\nu: B\rightarrow B\otimes B$ satisfying
$$
(\nu\otimes\id)\nu=(\id\otimes\nu)\nu=(\tau\otimes\id)(\id\otimes\nu)\nu.
$$
\end{defi}

One can check that $(B, \nu)$ is a perm coalgebra if and only if
$(B^{\ast}, \nu^{\ast})$ is a perm algebra.
For the dual version of Proposition \ref{pro:Lie+perm}, we have

\begin{pro}\label{pro:perm-coLie}
Let $(\g, \delta)$ be a Lie coalgebra and $(B, \nu)$ be a perm coalgebra.
Define a linear map $\vartheta: \g\otimes B\rightarrow(\g\otimes B)\otimes(\g\otimes B)$ by
$$
\vartheta(g\otimes b)=\delta(g)\bullet\nu(b)=\sum_{(g)}\sum_{(b)}
(g_{(1)}\otimes b_{(1)})\otimes(g_{(2)}\otimes b_{(2)}),
$$
for any $g\in\g$ and $b\in B$, where $\delta(g)=\sum_{(g)}g_{(1)}\otimes g_{(2)}$,
and $\nu(b)=\sum_{(b)}b_{(1)}\otimes b_{(2)}$ in the Sweedler notation.
Then $(\g\otimes B, \vartheta)$ is a Leibniz coalgebra.
\end{pro}

\begin{proof}
For any $\sum_{l}g'_{l}\otimes g''_{l}\otimes g'''_{l}\in \g\otimes\g\otimes\g$ and
$\sum_{k}b'_{k}\otimes b''_{k}\otimes b'''_{k}\in B\otimes B\otimes B$, we denote
$$
\Big(\sum_{l}g'_{l}\otimes g''_{l}\otimes g'''_{l}\Big)\bullet
\Big(\sum_{k}b'_{k}\otimes b''_{k}\otimes b'''_{k}\Big)
=\sum_{l}\sum_{k}(g'_{l}\otimes b'_{k})\otimes
(g''_{l}\otimes b''_{k})\otimes(g'''_{l}\otimes b'''_{k}).
$$
Then, by using the above notations, since $(B, \nu)$ is a perm coalgebra
and $(\g, \delta)$ is a Lie coalgebra, for any $g\otimes b\in\g\otimes B$, we have
\begin{align*}
&\;(\vartheta\otimes\id)(\vartheta(g\otimes b))
+(\tau\otimes\id)((\id\otimes\vartheta)(\vartheta(g\otimes b)))\\
=&\;(\delta\otimes\id)(\delta(g))\bullet(\nu\otimes\id)(\nu(b))
+(\tau\otimes\id)((\id\otimes\delta)(\delta(g)))\bullet
(\tau\otimes\id)((\id\otimes\nu)(\nu(b)))\\
=&\;(\id\otimes\delta)(\delta(g))\bullet(\id\otimes\nu)(\nu(b))\\
=&\;(\id\otimes\vartheta)(\vartheta(g\otimes b)).
\end{align*}
Thus, $(\g\otimes B, \vartheta)$ is a Leibniz coalgebra.
\end{proof}

Let $(\g, [-,-])$ be a Lie algebra. A bilinear form $\omega(-,-)$ on $(\g, [-,-])$
is called {\bf invariant} if it satisfies
$$
\omega([g_{1}, g_{2}],\; g_{3})=\omega(g_{1},\; [g_{2}, g_{3}]),
$$
for any $g_{1}, g_{2}, g_{3}\in\g$. A {\bf quadratic Lie algebra}, denoted by
$(\g, [-,-], \omega)$, is a Lie algebra $(\g, [-,-])$ together with a symmetric
invariant nondegenerate bilinear form $\omega(-,-)$.
Let $(\g, [-,-], \omega)$ be a quadratic Lie algebra. Then the bilinear form
$\omega(-,-)$ can naturally expand to the tensor product $\g\otimes\g\otimes
\cdots\otimes\g$, i.e.,
$$
\omega(-,-):\qquad (\underbrace{\g\otimes\cdots\otimes\g}_{\mbox{\tiny $k$-fold}})\otimes
(\underbrace{\g\otimes\cdots\otimes\g}_{\mbox{\tiny $k$-fold}})\longrightarrow\Bbbk,
$$
$\omega(g_{1}\otimes g_{2}\otimes\cdots\otimes g_{k},\ \ g'_{1}\otimes g'_{2}
\otimes\cdots\otimes g'_{k})=\prod_{i=1}^{k}\omega(g_{i}, g'_{i})$,
for any $g_{1}, g_{2},\cdots, g_{k}, g'_{1}, g'_{2},\cdots, g'_{k}\in\g$.
Then $\omega(-,-)$ on $\g\otimes\g\otimes\cdots\otimes\g$ is also a
nondegenerate bilinear form.

\begin{lem}\label{lem:Lie-dual}
Let $(\g, [-,-], \omega)$ be a quadratic Lie algebra. Define a linear map
$\delta_{\omega}: \g\rightarrow\g\otimes\g$ by $\omega(\delta_{\omega}(g_{1}),\;
g_{2}\otimes g_{3})=\omega(g_{1},\; [g_{2}, g_{3}])$, for any $g_{1}, g_{2}, g_{3}\in\g$.
Then $(\g, \delta_{\omega})$ is a Lie coalgebra.
\end{lem}

\begin{proof}
By the definition of $\delta_{\omega}$, it is easy to see that $\tau\delta_{\omega}
=-\delta_{\omega}$. Moreover, for any $g, g_{1}, g_{2}, g_{3}\in\g$, we get
\begin{align*}
&\;\omega\big((\id\otimes\delta_{\omega})(\delta_{\omega}(g))
-(\tau\otimes\id)((\id\otimes\delta_{\omega})(\delta_{\omega}(g))),\ \
g_{1}\otimes g_{2}\otimes g_{3}\big)\\
=&\;\omega\big(g,\ \ [g_{1}, [g_{2}, g_{3}]]-[g_{2}, [g_{1}, g_{3}]]\big)\\
=&\;\omega\big(g,\ \ [[g_{1}, g_{2}], g_{3}])\big)\\
=&\;\omega\big((\delta_{\omega}\otimes\id)(\delta_{\omega}(g)),\;
g_{1}\otimes g_{2}\otimes g_{3}\big).
\end{align*}
That is, $(\id\otimes\delta_{\omega})\delta_{\omega}-(\tau\otimes\id)
(\id\otimes\delta_{\omega})\delta_{\omega}=(\delta_{\omega}\otimes\id)\delta_{\omega}$.
Thus $(\g, \delta_{\omega})$ is a Lie coalgebra.
\end{proof}

\begin{ex}\label{ex:colib}
Consider the Lie algebra $\mathfrak{sl}_{2}$ is a 3-dimensional Lie algebra
with a basis $\{x, y, h\}$, where $[x, y]=h$, $[h, x]=2x$ and $[h, y]=-2y$. Let
$\omega(-,-)$ be one fourth of the Killing form on $\mathfrak{sl}_{2}(\Bbbk)$, i.e.,
$\omega(x, y)=1=\omega(y, x)$ and $\omega(h, h)=2$. Then $(\mathfrak{sl}_{2},
[-,-], \omega)$ is a quadratic Lie algebra. Define a coproduct $\delta_{\omega}$ on
$\mathfrak{sl}_{2}$ by $\omega(\delta_{\omega}(g_{1}),\; g_{2}\otimes g_{3})
=\omega(g_{1},\; [g_{2}, g_{3}])$, i.e.,
$$
\delta_{\omega}(x)=x\otimes h-h\otimes x,\qquad
\delta_{\omega}(y)=h\otimes y-y\otimes h,\qquad
\delta_{\omega}(h)=2y\otimes x-2x\otimes y,
$$
then we get a Lie coalgebra $(\mathfrak{sl}_{2}, \delta_{\omega})$.
\end{ex}

Recall that a {\bf bimodule $(V, \kkl, \kkr)$ over a perm algebra $(B, \diamond)$} is a
vector space $V$ with two linear maps $\kkl, \kkr: B\rightarrow\gl(V)$ such that
for any $b_{1}, b_{2}\in B$,
$$
\kkl(b_{1}\diamond b_{2})=\kkl(b_{1})\kkl(b_{2})=\kkl(b_{2})\kkl(b_{1}),\qquad\qquad
\kkr(b_{1}\diamond b_{2})=\kkr(b_{2})\kkr(b_{1})=\kkr(b_{2})\kkl(b_{1})=\kkl(b_{1})\kkr(b_{2}).
$$
In particular, $(B, \ffl_{B}, \ffr_{B})$ is a bimodule over $(B, \diamond)$, which is
called the regular bimodule, where $\ffl_{B}, \ffr_{B}: B\rightarrow\gl(B)$ are
given by $\ffl_{B}(b_{1})(b_{2})=b_{1}\diamond b_{2}$ and $\ffr_{B}(b_{1})(b_{2})
=b_{2}\diamond b_{1}$ for any $b_{1}, b_{2}\in B$.
Recall that a {\bf perm bialgebra} is a triple $(B, \diamond, \nu)$ such that
$(B, \diamond)$ is a perm algebra, $(B, \nu)$ is a perm coalgebra, and the
following compatibility condition holds:
\begin{align*}
&\qquad\qquad (\ffr_{B}(b_{1})\otimes\id)\nu(b_{2})
=\tau((\ffr_{B}(b_{2})\otimes\id)\nu(b_{1})),   \\
& \nu(b_{1}\diamond b_{2})=((\ffl_{B}-\ffr_{B})(b_{1})\otimes\id)\nu(b_{2})
+(\id\otimes\ffr_{B}(b_{2}))\nu(b_{1})  \\[-1mm]
&\qquad\quad\ \ \;\, =(\id\otimes\ffl_{B}(b_{1}))\nu(b_{2})
+((\ffl_{B}-\ffr_{B})(b_{2})\otimes\id)(\nu(b_{1})-\tau(\nu(b_{1})),
\end{align*}
for any $b_{1}, b_{2}\in B$. Following from Proposition \ref{pro:perm-coLie} and Lemma
\ref{lem:Lie-dual}, we can construct a Leibniz bialgebra from a perm bialgebra and a
quadratic Lie algebra as the following theorem.

\begin{thm}\label{thm:permbia-Lbi}
Let $(B, \diamond, \nu)$ be a perm bialgebra and $(\g, [-,-], \omega)$ be a quadratic
Lie algebra and $(\g\otimes B, \ast)$ be the induced Leibniz algebra by
$(\g, [-,-])$ and $(B, \diamond)$, i.e., $(g_{1}, b_{1})\ast(g_{2}, b_{2})
=[g_{1}, g_{2}]\otimes(b_{1}\diamond b_{2})$ for any $g_{1}, g_{2}\in\g$ and
$b_{1}, b_{2}\in B$. Define a linear map $\vartheta: \g\otimes B\rightarrow
(\g\otimes B)\otimes(\g\otimes B)$ by
\begin{align}
\vartheta(g\otimes b)=\delta_{\omega}(g)\bullet\nu(b))
:=\sum_{(g)}\sum_{(b)}(g_{(1)}\otimes b_{(1)})\otimes(g_{(2)}\otimes b_{(2)}),\label{copro}
\end{align}
for any $g\in\g$ and $b\in B$, where $\delta_{\omega}(g)=\sum_{(g)}g_{(1)}\otimes g_{(2)}$
and $\nu(b)=\sum_{(b)}b_{(1)}\otimes b_{(2)}$ in the Sweedler notation.
Then $(\g\otimes B, \ast, \vartheta)$ is a Leibniz bialgebra, which is called the
{\bf Leibniz bialgebra induced from $(B, \diamond, \nu)$ by $(\g, [-,-], \omega)$}.
\end{thm}

\begin{proof}
By Proposition \ref{pro:perm-coLie} and Lemma \ref{lem:Lie-dual}, we get
that $(\g\otimes B, \vartheta)$ is a Leibniz coalgebra. Thus, we only need to show
\begin{align*}
&\; \Phi:=\tau((\fr_{\g\otimes B}(g'\otimes b')\otimes\id)
(\vartheta(g\otimes b)))-(\fr_{\g\otimes B}(g\otimes b)\otimes\id)
(\vartheta(g'\otimes b'))=0,\\
&\; \Psi:=\vartheta((g\otimes b)\ast(g'\otimes b'))
-(\id\otimes\fr_{\g\otimes B}(g'\otimes b')
-(\fl_{\g\otimes B}+\fr_{\g\otimes B})(g'\otimes b')\otimes\id)((\id+\hat{\tau})
(\vartheta(g\otimes b)))\\[-1mm]
&\qquad\qquad\qquad\qquad-(\id\otimes\fl_{\g\otimes B}(g\otimes b)
+\fl_{\g\otimes B}(g\otimes b)\otimes\id)(\vartheta(g'\otimes b'))=0,
\end{align*}
for any $g, g'\in\g$ and $b, b'\in B$. In fact, for any $e, f\in\g$,
$$
\omega\Big(\sum_{(g)}g_{(2)}\otimes[g_{(1)}, g'],\ \
e\otimes f\Big)=\omega(g,\; [[g', f], e])=\omega\Big(\sum_{(g')}[g'_{(1)}, g]
\otimes g'_{(2)},\ \ e\otimes f\Big),
$$
we obtain
$$
\Phi=\Big(\sum_{(g)}g_{(2)}\otimes[g_{(1)}, g']\Big)\bullet
\Big((\ffr_{B}(b)\otimes\id)\nu(b')-\tau((\ffr_{B}(b')\otimes\id)\nu(b))\Big)=0,
$$
since $(B, \diamond, \nu)$ is a perm bialgebra. Moreover, for any $e, f\in\g$, since
$$
\omega\Big(\sum_{(g)}g_{(1)}\otimes[g_{(2)}, g'],\ \ e\otimes f\Big)
=\omega(g,\; [e, [g', f]])=-\omega(g,\; [[g', f], e])=
\omega\Big(\sum_{(g)}g_{(2)}\otimes[g_{(1)}, g'],\ \ e\otimes f\Big),
$$
we obtain $\sum_{(g)}g_{(1)}\otimes[g_{(2)}, g']=-\sum_{(g)}g_{(2)}\otimes[g_{(1)}, g']$.
Similarly, we also have $\sum_{(g)}g_{(1)}\otimes[g_{(2)}, g']=\sum_{(g')}[g, g_{(1)}]
\otimes g'_{(2)}$, $\sum_{(g)}[g_{(1)}, g']\otimes g_{(2)}=\sum_{(g)}[g', g_{(2)}]
\otimes g_{(1)}=\sum_{(g')}g'_{(1)}\otimes[g, g'_{(2)}]=-\sum_{(g)}[ g', g_{(1)}]
\otimes g_{(2)}=-\sum_{(g)}[g_{(2)}, g']\otimes g_{(1)}$ and $\delta_{\omega}([g, g'])
=\sum_{(g)}\big(g_{(1)}\otimes[g_{(2)}, g']+[g_{(1)}, g']\otimes g_{(2)}\big)$.
Therefore, we get
\begin{align*}
\Psi
=&\; \delta_{\omega}([g, g'])\bullet\nu(b\diamond b')\\
&\; -\sum_{(g)}\sum_{(b)}\Big(
(g_{(1)}\otimes[g_{(2)}, g'])\bullet(b_{(1)}\otimes(b_{(2)}\diamond b'))
-([g', g_{(1)}]\otimes g_{(2)})\bullet((b'\diamond b_{(1)})\otimes b_{(2)})\\[-3mm]
&\qquad\qquad\quad
-([g_{(1)}, g']\otimes g_{(2)}\bullet((b_{(1)}\diamond b')\otimes b_{(2)})
+(g_{(2)}\otimes[g_{(1)}, g'])\bullet(b_{(2)}\otimes(b_{(1)}\diamond b'))\\
&\qquad\qquad\quad
-([g', g_{(2)}]\otimes g_{(1)})\bullet((b'\diamond b_{(2)})\otimes b_{(1)})
-([g_{(2)}, g']\otimes g_{(1)}\bullet((b_{(2)}\diamond b')\otimes b_{(1)})\Big)\\
&\; -\sum_{(g')}\sum_{(b')}\Big(
(g'_{(1)}\otimes[g, g'_{(2)}])\bullet(b'_{(1)}\otimes(b\diamond b'_{(2)}))
+([g, g'_{(1)}]\otimes g'_{(2)})\otimes((b\diamond b'_{(1)})\otimes b'_{(2)})\Big)\\
=&\;\Big(\sum_{(g)}g_{(1)}\otimes[g_{(2)}, g']\Big)\bullet\Big(\nu(b\diamond b')
+\sum_{(b)}\big(b_{(2)}\otimes(b_{(1)}\diamond b')-b_{(1)}\otimes
(b_{(2)}\diamond b')\big)\\[-3mm]
&\qquad\qquad\qquad\qquad\qquad-\sum_{(b')}(b\diamond b'_{(1)})\otimes b'_{(2)}\Big)\\[-2mm]
&\;+\Big(\sum_{(g)}[g_{(1)}, g']\otimes g_{(2)}\Big)\bullet\Big(\nu(b\diamond b')
-\sum_{(b)}\big((b'\diamond b_{(1)})\otimes b_{(2)}
-(b_{(1)}\diamond b')\otimes b_{(2)}\\[-3mm]
&\qquad\qquad\qquad\qquad\qquad-(b'\diamond b_{(2)})\otimes b_{(1)}
+(b_{(2)}\diamond b')\otimes b_{(1)}\big)
-\sum_{(b')}b'_{(1)}\otimes(b\diamond b'_{(2)})\Big)\\
=&\; 0,
\end{align*}
since $(B, \diamond, \nu)$ is a perm bialgebra.
Thus, $(\g\otimes B, \ast, \vartheta)$ is a Leibniz bialgebra.
\end{proof}

Recall that a perm bialgebra $(B, \diamond, \nu)$ is called {\bf coboundary}
if there exists an element $r\in B\otimes B$ such that $\nu=\nu_{r}$, where
\begin{align}
\nu_{r}(b)=(\id\otimes\ffr_{B}(b)+(\ffr_{B}-\ffl_{B})(b)\otimes\id)(r), \label{permcobo}
\end{align}
for any $b\in B$. Let $(B, \diamond)$ be a perm algebra. An element
$r=\sum_{i}x_{i}\otimes y_{i}\in B\otimes B$ is said to be {\bf $(\ffl_{B},
\ffr_{B})$-invariant} if $\nu_{r}(b)=0$ for all $b\in B$. The equation
$$
\mathbf{P}_{r}:=r_{12}\diamond r_{23}-r_{13}\diamond r_{23}+r_{12}\diamond r_{13}
-r_{13}\diamond r_{12}=0
$$
is called the {\bf perm Yang-Baxter equation} (or $\PYBE$) in $(B, \diamond)$,
where $r_{12}\diamond r_{23}:=\sum_{i,j}x_{i}\otimes(y_{i}\diamond x_{j})\otimes y_{j}$,
$r_{13}\diamond r_{23}:=\sum_{i,j}x_{i}\otimes x_{j}\otimes(y_{i}\diamond y_{j})$,
$r_{12}\diamond r_{13}:=\sum_{i,j}(x_{i}\diamond x_{j})\otimes y_{i}\otimes y_{j}$ and
$r_{13}\diamond r_{12}:=\sum_{i,j}(x_{i}\diamond x_{j})\otimes y_{j}\otimes y_{i}$.

\begin{pro}[\cite{Lin}]\label{pro:perm-bia}
Let $(B, \diamond)$ be a perm algebra, $r\in B\otimes B$ and $\nu_{r}:
B\rightarrow B\otimes B$ be the linear map defined by Eq. \eqref{permcobo}.
\begin{enumerate}\itemsep=0pt
\item[$(i)$] If $r$ is a symmetric solution of the $\PYBE$ in $(B, \diamond)$, then
     $(B, \diamond, \nu_{r})$ is a perm bialgebra, which is called a {\bf triangular perm
     bialgebra} associated with $r$.
\item[$(ii)$] If $r$ is a solution of the $\PYBE$ in $(B, \diamond)$ and $r-\tau(r)$ is
     $(\ffl_{B}, \ffr_{B})$-invariant, then $(B, \diamond, \nu_{r})$ is a perm bialgebra,
     which is called a {\bf quasi-triangular perm bialgebra} associated with $r$.
\end{enumerate}
\end{pro}

\begin{ex}\label{ex:tri-permbi}
Consider 2-dimensional perm algebra $(B=\Bbbk\{e_{1}, e_{2}\}, \diamond)$, where
$e_{2}\diamond e_{1}=e_{1}$, $e_{2}\diamond e_{2}=e_{2}$. Then $r=e_{1}\otimes e_{1}$
is a symmetric solution of the $\PYBE$ in $(B, \diamond)$. We get a triangular perm
bialgebra $(B, \diamond, \nu_{r})$, where $\nu_{r}(e_{1})=0$ and $\nu_{r}(e_{2})
=-e_{1}\otimes e_{1}$.
\end{ex}

Let $(B, \diamond)$ be a perm algebra. For any $r\in B\otimes B$,
we can define a linear map $r^{\sharp}: B^{\ast}\rightarrow B$ by
$\langle r^{\sharp}(\xi_{1}),\; \xi_{2}\rangle=\langle\xi_{1}\otimes\xi_{2},\; r\rangle$
for any $\xi_{1}, \xi_{2}\in B^{\ast}$, and denote $\mathcal{I}=r^{\sharp}
-\tau(r)^{\sharp}: B^{\ast}\rightarrow B$.

\begin{defi}[\cite{Lin}]\label{def:fact-perm}
Let $(B, \diamond)$ be a perm algebra, $r\in B\otimes B$, and $(B, \diamond, \nu_{r})$
be a quasi-triangular perm bialgebra associated with $r$. If $\mathcal{I}=
r^{\sharp}-\tau(r)^{\sharp}: B^{\ast}\rightarrow B$ is an isomorphism of vector spaces,
then $(B, \diamond, \nu_{r})$ is called a {\bf factorizable perm bialgebra}.
\end{defi}

These special perm bialgebras are all related to the solutions of the Yang-Baxter equation.
Next, we consider the relation between the solutions of the $\PYBE$ in a perm
algebra and the solutions of the $\LYBE$ in the induced Leibniz algebra.
Let $(\g, [-,-], \omega)$ be a quadratic Lie algebra and $\{e_{1}, e_{2},\cdots, e_{n}\}$
be a basis of $\g$. Since $\omega(-,-)$ is symmetric nondegenerate, we get a basis
$\{f_{1}, f_{2},\cdots, f_{n}\}$ of $\g$, which is called the dual basis of $\{e_{1},
e_{2},\cdots, e_{n}\}$ with respect to $\omega(-,-)$, by $\omega(f_{i}, e_{j})=\delta_{ij}$,
where $\delta_{ij}$ is the Kronecker delta. Here, for some special solutions of the
$\PYBE$ in a perm algebra, we have

\begin{pro}\label{pro:PYBE-LYBE}
Let $(B, \diamond)$ be a perm algebra, $(\g, [-,-], \omega)$ be a quadratic Lie algebra
and $(\g\otimes B, \ast)$ be the induced Leibniz algebra. Suppose that $r=\sum_{i}
x_{i}\otimes y_{i}\in B\otimes B$ is a solution of the $\PYBE$ in $(B, \diamond)$,
$r-\tau(r)$ is $(\ffl_{B}, \ffr_{B})$-invariant, $\{e_{1}, e_{2},\cdots, e_{n}\}$ is
a basis of $\g$ and $\{f_{1}, f_{2},\cdots, f_{n}\}$ is the dual basis of $\{e_{1},
e_{2},\cdots, e_{n}\}$ with respect to $\omega(-,-)$. Then
\begin{align}
\widetilde{r}=\sum_{i, j}(e_{j}\otimes x_{i})\otimes(f_{j}\otimes y_{i})
\in(\g\otimes B)\otimes(\g\otimes B)  \label{r-max}
\end{align}
is a solution of the $\LYBE$ in $(\g\otimes B, \ast)$, and $\widetilde{r}-\tau(\widetilde{r})$
is Leib-invariant.

In particular, $\widetilde{r}$ is a symmetric solution of the $\LYBE$ in $(\g\otimes B, \ast)$
if $r$ is a symmetric solution of the $\PYBE$ in $(B, \diamond)$.
\end{pro}

\begin{proof}
First, note that
\begin{align*}
&\;\widetilde{r}_{12}\ast\widetilde{r}_{13}-\widetilde{r}_{12}\ast\widetilde{r}_{23}
-\widetilde{r}_{23}\ast\widetilde{r}_{12}+\widetilde{r}_{23}\ast\widetilde{r}_{13}\\
=&\;\sum_{i,j}\sum_{k,l}\Big(\big([e_{k}, e_{l}]\otimes f_{k}\otimes f_{l}\big)
\bullet\big((x_{i}\diamond x_{j})\otimes y_{i}\otimes y_{j}\big)
-\big(e_{k}\otimes[f_{k}, e_{l}]\otimes f_{l}\big)\bullet
\big(x_{i}\otimes(y_{i}\diamond x_{j})\otimes y_{j}\big)\\[-4mm]
&\qquad\quad-\big(e_{k}\otimes[e_{l}, f_{k}]\otimes f_{l}\big)
\bullet\big(x_{i}\otimes(x_{j}\diamond y_{i})\otimes y_{j}\big)
+\big(e_{k}\otimes e_{l}\otimes[f_{l}, f_{k}]\big)\bullet
\big(x_{i}\otimes x_{j}\otimes(y_{j}\diamond y_{i})\big)\Big),
\end{align*}
$\sum_{k,l}e_{k}\otimes[f_{k}, e_{l}]\otimes f_{l}=-\sum_{k,l}e_{k}\otimes[e_{l}, f_{k}]
\otimes f_{l}$ and for any $1\leq s, u, v\leq n$,
\begin{align*}
\omega\Big(\sum_{k,l}[e_{k}, e_{l}]\otimes f_{k}\otimes f_{l},\ \
e_{s}\otimes e_{u}\otimes e_{v}\Big)&=\omega([e_{u}, e_{v}],\; e_{s}),\\[-2mm]
\omega\Big(\sum_{k,l}(e_{k}\otimes[f_{k}, e_{l}]\otimes f_{l},\ \
e_{s}\otimes e_{u}\otimes e_{v}\Big)&=\omega([e_{v}, e_{u}],\; e_{s}),\\[-2mm]
\omega\Big(\sum_{k,l}e_{k}\otimes e_{l}\otimes[f_{l}, f_{k}],\ \
e_{s}\otimes e_{u}\otimes e_{v}\Big)&=\omega([e_{v}, e_{u}],\; e_{s}).
\end{align*}
By the nondegeneracy of $\omega(-,-)$, we get $\sum_{k,l}[e_{k}, e_{l}]\otimes f_{k}
\otimes f_{l}=\sum_{k,l}e_{k}\otimes[e_{l}, f_{k}]\otimes f_{l}=-\sum_{k,l}e_{k}\otimes
[f_{k}, e_{l}]\otimes f_{l}=-\sum_{k,l}e_{k}\otimes e_{l}\otimes[f_{l}, f_{k}]$, and so that
\begin{align*}
&\;\widetilde{r}_{12}\ast\widetilde{r}_{13}-\widetilde{r}_{12}\ast\widetilde{r}_{23}
-\widetilde{r}_{23}\ast\widetilde{r}_{12}+\widetilde{r}_{23}\ast\widetilde{r}_{13}\\
=&\;\sum_{i,j}\sum_{k,l}\big([e_{k}, e_{l}]\otimes f_{k}\otimes f_{l}\big)
\bullet\Big((x_{i}\diamond x_{j})\otimes y_{i}\otimes y_{j}
+x_{i}\otimes(y_{i}\diamond x_{j})\otimes y_{j}\\[-5mm]
&\qquad\qquad\qquad\qquad\qquad\qquad-x_{i}\otimes(x_{j}\diamond y_{i})\otimes y_{j}
-x_{i}\otimes x_{j}\otimes(y_{j}\diamond y_{i})\Big).
\end{align*}
If $r-\tau(r)$ is $(\ffl_{B}, \ffr_{B})$-invariant, i.e., $\sum_{i,j}\big(
x_{i}\otimes(y_{i}\diamond x_{j})\otimes y_{j}+(x_{i}\diamond x_{j})\otimes y_{i}
\otimes y_{j}-(x_{j}\diamond x_{i})\otimes y_{i}\otimes y_{j}-y_{i}\otimes
(x_{i}\diamond x_{j})\otimes y_{j}-(y_{i}\diamond x_{j})\otimes x_{i}\otimes y_{j}
+(x_{j}\diamond y_{i})\otimes x_{i}\otimes y_{j}\big)=0$, we obtain
\begin{align*}
\mathbf{P}_{r}&=\sum_{i,j}\Big(x_{i}\otimes(y_{i}\diamond x_{j})\otimes y_{j}
-x_{i}\otimes x_{j}\otimes(y_{i}\diamond y_{j})+(x_{i}\diamond x_{j})\otimes
y_{i}\otimes y_{j}-(x_{i}\diamond x_{j})\otimes y_{j}\otimes y_{i}\Big)\\[-2mm]
&=\sum_{i,j}\Big(y_{i}\otimes(x_{i}\diamond x_{j})\otimes y_{j}
+(y_{i}\diamond x_{j})\otimes x_{i}\otimes y_{j}
-(x_{j}\diamond y_{i})\otimes x_{i}\otimes y_{j}
-x_{j}\otimes x_{i}\otimes(y_{j}\diamond y_{i})\Big).
\end{align*}
Therefore, we get
$\widetilde{r}_{12}\ast\widetilde{r}_{13}-\widetilde{r}_{12}\ast\widetilde{r}_{23}
-\widetilde{r}_{23}\ast\widetilde{r}_{12}+\widetilde{r}_{23}\ast\widetilde{r}_{13}
=\sum_{k,l}\big([e_{k}, e_{l}]\otimes f_{k}\otimes f_{l}\big)\bullet\big(
(\tau\otimes\id)(\mathbf{P}_{r})\big)$. Thus, $\widetilde{r}$ is a solution
of the $\LYBE$ in $(\g\otimes B, \ast)$ if a solution of the $\PYBE$ in $(B, \diamond)$
and $r-\tau(r)$ is $(\ffl_{B}, \ffr_{B})$-invariant.

Second, for any $e_{s}, e_{t}\in\g$, since
$$
\omega\Big(\sum_{j}[g, e_{j}]\otimes f_{j},\ \ e_{s}\otimes e_{t}\Big)
=\omega(g,\; [e_{t}, e_{s}])=-\omega(e_{t},\; [g, e_{s}])
=-\omega\Big(\sum_{j}[e_{j}, g]\otimes f_{j},\ \ e_{s}\otimes e_{t}\Big),
$$
by the nondegeneracy of $\omega(-,-)$, we obtain $\sum_{j}[g, e_{j}]\otimes f_{j}
=-\sum_{j}[e_{j}, g]\otimes f_{j}$. Similarly, we also have $\sum_{j}[g, e_{j}]\otimes f_{j}
=\sum_{j}e_{j}\otimes[f_{j}, g]=\sum_{j}[g, f_{j}]\otimes e_{j}=\sum_{j}f_{j}\otimes[e_{j}, g]
=-\sum_{j}[f_{j}, g]\otimes e_{j}$.
Thus, for any $g\in\g$ and $b\in B$, we get
\begin{align*}
&\;\Big((\fl_{\g\otimes B}+\fr_{\g\otimes B})(g\otimes b)\otimes\id)
-\id\otimes\fr_{\g\otimes B}(g\otimes b)\Big)(\widetilde{r}-\tau(\widetilde{r}))\\
=&\;\sum_{i,j}\Big(([g, e_{j}]\otimes f_{j})\bullet((b\diamond x_{i})\otimes y_{i})
+([e_{j}, g]\otimes f_{j})\bullet((x_{j}\diamond b)\otimes y_{i})\\[-3mm]
&\qquad\quad-(e_{j}\otimes[f_{j}, g])\bullet(x_{i}\otimes(y_{i}\diamond b))
-([g, f_{j}]\otimes e_{j})\bullet((b\diamond y_{i})\otimes x_{i})\\[-1mm]
&\qquad\quad-([f_{j}, g]\otimes e_{j})\bullet((y_{i}\diamond b)\otimes x_{i})
+(f_{j}\otimes[e_{j}, g])\bullet(y_{i}\otimes(x_{i}\diamond b))\Big)\\
=&\;\sum_{i,j}([g, e_{j}]\otimes f_{j})\bullet\Big((b\diamond x_{i})\otimes y_{i}
-(x_{j}\diamond b)\otimes y_{i}-x_{i}\otimes(y_{i}\diamond b))\\[-5mm]
&\qquad\qquad\qquad\qquad\qquad
-(b\diamond y_{i})\otimes x_{i}+(y_{i}\diamond b)\otimes x_{i})
+y_{i}\otimes(x_{i}\diamond b)\Big)\\
=&\; 0,
\end{align*}
if $r-\tau(r)$ is $(\ffl_{B}, \ffr_{B})$-invariant, i.e., $\sum_{i}\big(x_{i}\otimes
(y_{i}\diamond b)+(x_{i}\diamond b)\otimes y_{i}-(b\diamond x_{i})\otimes y_{i}
-y_{i}\otimes(x_{i}\diamond b)-(y_{i}\diamond b)\otimes x_{i}+(b\diamond y_{i})
\otimes x_{i}\big)=0$. That is, $\widetilde{r}$ is a solution
of the $\LYBE$ in $(\g\otimes B, \ast)$ and $\widetilde{r}-\tau(\widetilde{r})$
is Leib-invariant if $r$ is a solution of the $\PYBE$ in $(B, \diamond)$ and $r-\tau(r)$
is $(\ffl_{B}, \ffr_{B})$-invariant.

Finally, note that for any $e_{s}, e_{t}\in\g$,
$$
\omega\Big(\sum_{j}e_{j}\otimes f_{j},\; e_{s}\otimes e_{t}\Big)
=\omega(e_{s}, e_{t})=\omega(e_{t}, e_{s})
=\omega\Big(\sum_{j}f_{j}\otimes e_{j},\; e_{s}\otimes e_{t}\Big).
$$
That is, $\sum_{j}e_{j}\otimes f_{j}$ is symmetric.
We get $\widetilde{r}$ is symmetric if $r$ is symmetric. The proof is finished.
\end{proof}

We now give another main conclusion of this section.

\begin{thm}\label{thm:indu-sLbia}
Let $(B, \diamond, \nu)$ be a perm bialgebra and $(\g, [-,-], \omega)$ be a quadratic
Lie algebra, and $(\g\otimes B, \ast, \vartheta)$ be the induced Leibniz bialgebra from
$(B, \diamond, \nu)$ by $(\g, [-,-], \omega)$. If $\nu=\nu_{r}$ is defined by Eq.
\eqref{permcobo} for $r\in B\otimes B$, then $(\g\otimes B, \ast, \vartheta)=
(\g\otimes B, \ast, \vartheta_{\widetilde{r}})$ as Leibniz bialgebras, where
$\vartheta_{\widetilde{r}}$ is defined by Eq. \eqref{cobLeb} and
$\widetilde{r}$ is defined by Eq. \eqref{r-max}. Therefore, we obtain
\begin{enumerate}\itemsep=0pt
\item[$(i)$]  $(\g\otimes B, \ast, \vartheta)$ is coboundary if
     $(B, \diamond, \nu)$ is coboundary;
\item[$(ii)$] $(\g\otimes B, \ast, \vartheta)$ is quasi-triangular if
     $(B, \diamond, \nu)$ is quasi-triangular;
\item[$(iii)$] $(\g\otimes B, \ast, \vartheta)$ is triangular if
     $(B, \diamond, \nu)$ is triangular;
\item[$(iv)$] $(\g\otimes B, \ast, \vartheta)$ is factorizable if
     $(B, \diamond, \nu)$ is factorizable.
\end{enumerate}
\end{thm}

\begin{proof}
Let $r=\sum_{i}x_{i}\otimes y_{i}\in B\otimes B$. First, for any $g\in\g$ and $b\in B$, we have
$$
\vartheta(g\otimes b)=\Big(\sum_{(g)}g_{(1)}\otimes g_{(2)}\Big)\bullet\Big(\sum_{i}
\big(x_{i}\otimes(y_{i}\diamond b)+(x_{i}\diamond b)\otimes y_{i}
-(b\diamond x_{i})\otimes y_{i}\big)\Big),
$$
where $\delta_{\omega}(g)=\sum_{(g)}g_{(1)}\otimes g_{(2)}$ and $\nu_{r}(b)=
(\id\otimes\ffr_{B}(b)+(\ffr_{B}-\ffl_{B})(b)\otimes\id)(r)=\sum_{i}\big(x_{i}
\otimes(y_{i}\diamond b)+(x_{i}\diamond b)\otimes y_{i}-(b\diamond x_{i})\otimes
y_{i}\big)$. On the other hand,
\begin{align*}
\vartheta_{\widetilde{r}}(g\otimes b)&=\big((\fl_{\g\otimes B}+\fr_{\g\otimes B})
(g\otimes b)\otimes\id)-\id\otimes\fr_{\g\otimes B}(g\otimes b)\big)(\widetilde{r})\\
&=\sum_{i,j}\Big(\big([g, e_{j}]\otimes f_{j}\big)\bullet\big((b\diamond x_{i})\otimes
y_{i}-(x_{i}\diamond b)\otimes y_{i}\big)-\big(e_{j}\otimes[f_{j}, g]\big)\bullet
\big(x_{i}\otimes(y_{i}\diamond b)\big)\Big),
\end{align*}
where $\{e_{1}, e_{2},\cdots, e_{n}\}$ is a basis of $\g$ and $\{f_{1}, f_{2},\cdots,
f_{n}\}$ is the dual basis of $\{e_{1}, e_{2}, \cdots, e_{n}\}$ with respect to $\omega(-,-)$.
For any basis elements $e_{s}, e_{t}\in\g$, since
\begin{align*}
&\qquad\qquad\qquad\omega\Big(\sum_{(g)}g_{(1)}\otimes g_{(2)},\ \
e_{s}\otimes e_{t}\Big)=\omega(g,\; [e_{s}, e_{t}]),\\[-2mm]
& \omega\Big(\sum_{j}[g, e_{j}]\otimes f_{j},\ \ e_{s}\otimes e_{t}\Big)
=\omega(g,\; [e_{t}, e_{s}])=\omega\Big(\sum_{j}e_{j}\otimes[f_{j}, g],\ \
e_{s}\otimes e_{t}\Big),
\end{align*}
we get $\sum_{(g)}g_{(1)}\otimes g_{(2)}=-\sum_{j}[g, e_{j}]\otimes f_{j}=-
\sum_{j}e_{j}\otimes[f_{j}, g]$ since $\omega(-,-)$ is nondegenerate.
Thus, $\vartheta_{\widetilde{r}}(g\otimes b)=\vartheta(g\otimes b)$ for any
$g\in\g$ and $b\in B$, and so that $(\g\otimes B, \ast, \vartheta)=(\g\otimes B,
\ast, \vartheta_{\widetilde{r}})$ as Leibniz bialgebras.

Second, $(ii)$ and $(iii)$ follow from Proposition \ref{pro:PYBE-LYBE}.
Finally, suppose $(B, \diamond, \nu_{r})$ is factorizable. We get $\mathcal{I}=
r^{\sharp}-\tau(r)^{\sharp}: B^{\ast}\rightarrow B$ is an isomorphism of
vector spaces. We need to show that $\widetilde{\mathcal{I}}:=\widetilde{r}^{\sharp}
-\tau(\widetilde{r})^{\sharp}: (\g\otimes B)^{\ast}\rightarrow \g\otimes B$ is an
isomorphism of vector spaces. Denote $\kappa:=\sum_{j}e_{j}\otimes
f_{j}\in\g\otimes\g$. Define $\kappa^{\sharp}: \g^{\ast}\rightarrow\g$ by
$\langle\kappa^{\sharp}(\xi_{1}),\; \xi_{2}\rangle=\langle\xi_{1}
\otimes\xi_{2},\; \kappa\rangle$, for any $\xi_{1}, \xi_{2}\in\g^{\ast}$.
Then, one can check that $\kappa^{\sharp}$ is a linear isomorphism and
$\langle\kappa^{\sharp}(\xi_{1}),\; \xi_{2}\rangle=\langle\xi_{1}\otimes\xi_{2},\;
\kappa\rangle=\langle\xi_{2}\otimes\xi_{1},\; \kappa\rangle
=\langle\kappa^{\sharp}(\xi_{2}),\; \xi_{1}\rangle$.
Therefore, for any $\xi_{1}, \xi_{2}\in\g^{\ast}$ and $\eta_{1}, \eta_{2}\in B^{\ast}$,
\begin{align*}
\langle\widetilde{r}^{\sharp}(\xi_{1}\otimes\eta_{1}),\; \xi_{2}\otimes\eta_{2}\rangle
&=\sum_{i,j}\langle(\xi_{1}\otimes\eta_{1})\otimes(\xi_{2}\otimes\eta_{2}),\ \
(e_{j}\otimes x_{i})\otimes(f_{j}\otimes y_{i})\rangle\\[-2mm]
&=\Big(\sum_{i}\langle\xi_{1}, e_{j}\rangle\langle\xi_{2}, f_{j}\rangle\Big)
\Big(\sum_{j}\langle\eta_{1}, x_{i}\rangle\langle\eta_{2}, y_{i}\rangle\Big)\\[-2mm]
&=\langle\kappa^{\sharp}(\xi_{1}),\; \xi_{2}\rangle
\langle r^{\sharp}(\eta_{1}),\; \eta_{2}\rangle\\
&=\langle\kappa^{\sharp}(\xi_{1})\otimes r^{\sharp}(\eta_{1}),\; \xi_{2}\otimes\eta_{2}\rangle.
\end{align*}
That is, $\widetilde{r}^{\sharp}=\kappa^{\sharp}\otimes r^{\sharp}$, Similarly,
$\tau(\widetilde{r})^{\sharp}=\kappa^{\sharp}\otimes\tau(r)^{\sharp}$. Thus,
$\widetilde{\mathcal{I}}=\kappa^{\sharp}\otimes\mathcal{I}$ is an isomorphism of
vector spaces. The proof is completed.
\end{proof}

Let $(B, \diamond)$ be a perm algebra, $r\in B\otimes B$ and $(\g, [-,-], \omega)$ be a
quadratic Lie algebra. By the proof of Theorem \ref{thm:indu-sLbia},
we have following commutative diagram:
$$
\xymatrix@C=3cm@R=0.6cm{
\txt{$r$ \\ {\tiny a solution of the $\PYBE$ in $(B, \diamond)$}\\
{\tiny such that $r-\tau(r)$ is $(\ffl_{B}, \ffr_{B})$-invariant}}
\ar[d]_{{\rm Pro.}~\ref{pro:PYBE-LYBE}}\ar[r]^{{\rm Pro.}~\ref{pro:perm-bia}} &
\txt{$(B, \diamond, \vartheta_{r})$ \\ {\tiny a quasi-triangular perm bialgebra}}
\ar[d]^{{\rm Thm.}~\ref{thm:indu-sLbia}} \\
\txt{$\widetilde{r}$ \\ {\tiny a solution of the $\LYBE$ in $(\g\otimes B, \ast)$}\\
{\tiny such that $\widetilde{r}-\tau(\widetilde{r})$ is Leib-invariant}}
\ar[r]^{{\rm Pro.}~\ref{pro:lie-bia}} &
\txt{$(\g\otimes B, \ast, \vartheta)=(\g\otimes B, \ast, \vartheta_{\widetilde{r}})$ \\
{\tiny the induced quasi-triangular Leibniz bialgebra}}}
$$
In particular, for a symmetric solution $r$ of the $\PYBE$ in $(B, \diamond)$, we get
$\widetilde{r}$ is a symmetric solution of the $\LYBE$ in $(\g\otimes B, \ast)$,
$(\g\otimes B, \ast, \vartheta)=(\g\otimes B, \ast, \vartheta_{\widetilde{r}})$
as triangular Leibniz bialgebras, and the diagram is also commutative.

In \cite{Kup}, Kupershmidt found that the $\CYBE$ in tensor form on Lie algebras can
be converted into an $\mathcal{O}$-operator associated to the coadjoint representation.
This conclusion has been confirmed on various types of algebras.
Let $(A, \ast)$ be a Leibniz algebra and $(V, \kl, \kr)$ be a representation of
$(A, \ast)$. Recall that a linear map $T: V\rightarrow A$ is called an {\bf
$\mathcal{O}$-operator of $(A, \ast)$ associated to $(V, \kl, \kr)$} if for
any $v_{1}, v_{2}\in V$,
$$
T(v_{1})\ast T(v_{2})=T\big(\kl(T(v_{1}))(v_{2})+\kr(T(v_{2}))(v_{1})\big).
$$

\begin{pro}[\cite{TS,BLST}]\label{pro:o-leib}
Let $(A, \ast)$ be a Leibniz algebra, $r\in A\otimes A$ be symmetric.
Then $r$ is a solution of the $\LYBE$ in $(A, \ast)$ if and only if
$r^{\sharp}: A^{\ast}\rightarrow A$ is an $\mathcal{O}$-operator of $(A, \ast)$
associated to the coregular representation $(A^{\ast}, \fl_{A}^{\ast}, -\fl_{A}^{\ast}-\fr_{A}^{\ast})$.
\end{pro}

The $\mathcal{O}$-operator of a perm algebra was considered in \cite{LZB}.
Recall that an {\bf $\mathcal{O}$-operator of a perm algebra $(B, \diamond)$
associated to a bimodule $(V, \kl, \kr)$} is a linear map $T: V\rightarrow B$ such that
$$
T(v_{1})\diamond T(v_{2})=T\big(\kl(T(v_{1}))(v_{2})+\kr(T(v_{2}))(v_{1})\big),
$$
for any $v_{1}, v_{2}\in V$.

\begin{pro}[\cite{LZB}]\label{pro:o-perm}
Let $(B, \diamond)$ be a perm algebra, $r\in B\otimes B$ be symmetric.
Then $r$ is a solution of the $\PYBE$ in $(B, \diamond)$ if and only if $r^{\sharp}$
is an $\mathcal{O}$-operator of $(B, \diamond)$ associated to the coregular
bimodule $(B^{\ast}, -\ffl_{B}^{\ast}, \ffr_{B}^{\ast}-\ffl_{B}^{\ast})$.
\end{pro}

Let $(B, \diamond)$ be a perm algebra, $r\in B\otimes B$ and $(\g, [-,-], \omega)$ be a
quadratic Lie algebra. By the proof of Theorem \ref{thm:indu-sLbia}, we also have
the following commutative diagram:
$$
\xymatrix@C=3cm@R=0.5cm{
\txt{$r$ \\ {\tiny a symmetric solution} \\ {\tiny of the $\PYBE$ in $(B, \diamond)$}}
\ar[d]_-{{\rm Pro.}~\ref{pro:PYBE-LYBE}}\ar[r]^-{{\rm Pro.}~\ref{pro:o-perm}} &
\txt{$r^{\sharp}$\\ {\tiny an $\mathcal{O}$-operator of $(B, \diamond)$} \\
{\tiny associated to $(B^{\ast}, -\ffl_{B}^{\ast}, \ffr_{B}^{\ast}-\ffl_{B}^{\ast})$}}
\ar[d]^-{\mbox{$\kappa^{\sharp}\otimes-$}} \\
\txt{$\widetilde{r}$ \\ {\tiny a symmetric solution} \\ {\tiny of the $\LYBE$ in
$(\g\otimes B, \ast)$}} \ar[r]^-{{\rm Pro.}~\ref{pro:o-leib}}
& \txt{$\widetilde{r}^{\sharp}=\kappa^{\sharp}\otimes r^{\sharp}$ \\
{\tiny an $\mathcal{O}$-operator of $(\g\otimes B, \ast)$ } \\
{\tiny associated to $((\g\otimes B)^{\ast}, \fl_{\g\otimes B}^{\ast},
-\fl_{\g\otimes B}^{\ast}-\fr_{\g\otimes B}^{\ast})$}}}
$$

At the end of this section, we give a simple example.

\begin{ex}\label{ex:ind-tri}
Let $(B=\Bbbk\{e_{1}, e_{2}\}, \diamond, \nu_{r})$ be the 2-dimensional perm bialgebra
given in Example \ref{ex:tri-permbi}, i.e., $e_{2}\diamond e_{1}=e_{1}$,
$e_{2}\diamond e_{2}=e_{2}$, $\nu_{r}(e_{1})=0$ and $\nu_{r}(e_{2})=-e_{1}\otimes e_{1}$,
and let $(\mathfrak{sl}_{2}, [-,-], \omega)$ be the quadratic Lie algebra
defined in Example \ref{ex:colib}. By Theorem \ref{thm:permbia-Lbi}, we obtain a
Leibniz bialgebra $(\mathfrak{sl}_{2}\otimes B, \ast, \vartheta)$, where
\begin{align*}
&(x\otimes e_{2})\ast(y\otimes e_{1})=h\otimes e_{1}=-(y\otimes e_{2})\ast(x\otimes e_{1}),\\
&(x\otimes e_{2})\ast(y\otimes e_{2})=h\otimes e_{2}=-(y\otimes e_{2})\ast(x\otimes e_{2}),\\
&(h\otimes e_{2})\ast(x\otimes e_{1})=2x\otimes e_{1}=-(x\otimes e_{2})\ast(h\otimes e_{1}),\\
&(h\otimes e_{2})\ast(x\otimes e_{2})=2x\otimes e_{2}=-(x\otimes e_{2})\ast(h\otimes e_{2}),\\
&(y\otimes e_{2})\ast(h\otimes e_{1})=2y\otimes e_{1}=-(h\otimes e_{2})\ast(y\otimes e_{1}),\\
&(y\otimes e_{2})\ast(h\otimes e_{2})=2y\otimes e_{2}=-(h\otimes e_{2})\ast(y\otimes e_{2}),\\
&\vartheta(x\otimes e_{2})=(h\otimes e_{1})\otimes(x\otimes e_{1})
-(x\otimes e_{1})\otimes(h\otimes e_{1}),\\
&\vartheta(y\otimes e_{2})=(y\otimes e_{1})\otimes(h\otimes e_{1})
-(h\otimes e_{1})\otimes(y\otimes e_{1}),\\
&\vartheta(h\otimes e_{2})=2(x\otimes e_{1})\otimes(y\otimes e_{1})
-2(y\otimes e_{1})\otimes(x\otimes e_{1}),
\end{align*}
and others are all zero. The symmetric solution $r=e_{1}\otimes e_{1}$ of the $\PYBE$
in $(B, \diamond)$ induces an $\mathcal{O}$-operator $r^{\sharp}: B^{\ast}\rightarrow
B$, which is given by $r^{\sharp}(\xi_{1})=e_{1}$ and $r^{\sharp}(\xi_{2})=0$,
where $\xi_{1}, \xi_{2}\in B^{\ast}$ is the dual basis of $e_{1}, e_{2}$.
In the quadratic Lie algebra $(\mathfrak{sl}_{2}, [-,-], \omega)$, note that
$\{y, x, \tfrac{1}{2}h\}$ is the dual basis of $\{x, y, h\}$ with respect to $\omega(-,-)$.
Denote $\kappa:=x\otimes y+y\otimes x+\tfrac{1}{2}h\otimes h$ and by $\{\xi_{x},
\xi_{y}, \xi_{h}\}$ the dual basis in $\mathfrak{sl}_{2}^{\ast}$ of $\{x, y, h\}$.
Define $\kappa^{\sharp}: \mathfrak{sl}_{2}^{\ast}\rightarrow\mathfrak{sl}_{2}$ by
$\langle\kappa^{\sharp}(\eta_{1}),\; \eta_{2}\rangle=\langle\eta_{1}\otimes\eta_{2},\;
\kappa\rangle$ for any $\eta_{1}, \eta_{2}\in\mathfrak{sl}_{2}^{\ast}$. Then we get
a linear map $\kappa^{\sharp}\otimes r^{\sharp}: (\mathfrak{sl}_{2}\otimes B)^{\ast}
\rightarrow\mathfrak{sl}_{2}\otimes B$:
$$
(\kappa^{\sharp}\otimes r^{\sharp})(\xi_{x}\otimes\xi_{1})=y\otimes e_{1},\qquad
(\kappa^{\sharp}\otimes r^{\sharp})(\xi_{y}\otimes\xi_{1})=x\otimes e_{1},\qquad
(\kappa^{\sharp}\otimes r^{\sharp})(\xi_{h}\otimes\xi_{1})=\tfrac{1}{2}h\otimes e_{1},
$$
and $(\kappa^{\sharp}\otimes r^{\sharp})(\xi_{x}\otimes\xi_{2})=
(\kappa^{\sharp}\otimes r^{\sharp})(\xi_{y}\otimes\xi_{2})=
(\kappa^{\sharp}\otimes r^{\sharp})(\xi_{h}\otimes\xi_{2})=0$.
One can check that $\kappa^{\sharp}\otimes r^{\sharp}$ is an $\mathcal{O}$-operator
of $(\g\otimes B, \ast)$ associated to the coregular representation $((\g\otimes B)^{\ast},
\fl_{\g\otimes B}^{\ast}, -\fl_{\g\otimes B}^{\ast}-\fr_{\g\otimes B}^{\ast})$.

On the other hand, one can check that
$$
\widetilde{r}=(x\otimes e_{1})\otimes(y\otimes e_{1})
+(y\otimes e_{1})\otimes(x\otimes e_{1})+\tfrac{1}{2}(h\otimes e_{1})\otimes(h\otimes e_{1})
$$
is a symmetric solution of the $\LYBE$ in $(\g\otimes B, \ast)$. By direct calculation,
we obtain $\vartheta_{\widetilde{r}}=\vartheta$,
$$
\widetilde{r}^{\sharp}(\xi_{x}\otimes\xi_{1})=y\otimes e_{1},\qquad
\widetilde{r}^{\sharp}(\xi_{y}\otimes\xi_{1})=x\otimes e_{1},\qquad
\widetilde{r}^{\sharp}(\xi_{h}\otimes\xi_{1})=\tfrac{1}{2}h\otimes e_{1},
$$
and $\widetilde{r}^{\sharp}(\xi_{x}\otimes\xi_{2})=\widetilde{r}^{\sharp}(\xi_{y}
\otimes\xi_{2})=\widetilde{r}^{\sharp}(\xi_{h}\otimes\xi_{2})=0$. That is,
$\widetilde{r}^{\sharp}=\kappa^{\sharp}\otimes r^{\sharp}$.
\end{ex}

\smallskip
\section{Infinite-dimensional Leibniz bialgebras from Lie bialgebras} \label{sec:infi-bia}
In this section, we recall the notion of a quadratic $\bz$-graded perm algebra, as
a $\bz$-graded perm algebra equipped with an invariant bilinear form. We introduce
the notion of completed tensor product and show that the tensor product of
a finite-dimensional Lie bialgebra and a quadratic $\bz$-graded perm algebra can be
naturally endowed with an infinite-dimensional Leibniz bialgebra.

\begin{defi}\label{def:zgrad-alg}
A {\bf $\bz$-graded Leibniz algebra} (resp. {\bf $\bz$-graded perm algebra})
is a Leibniz algebra $(A, \ast)$ (resp. a perm algebra $(B, \diamond)$) with a
linear decomposition $A=\oplus_{i\in\bz}A_{i}$ (resp. $B=\oplus_{i\in\bz}B_{i}$) such that
each $A_{i}$ (resp. $B_{i}$) is finite-dimensional, $A_{i}\ast A_{j}\subseteq A_{i+j}$
(resp. $B_{i}\diamond B_{j}\subseteq B_{i+j}$) for all $i, j\in\bz$.
\end{defi}

\begin{ex}[\cite{LZB}]\label{ex:grperm}
Let $B=\{f_{1}\partial_{1}+f_{2}\partial_{2}\mid f_{1}, f_{2}\in\Bbbk[x_{1}^{\pm},
x_{2}^{\pm}]\}$ and define a binary operation $\diamond: B\otimes B\rightarrow B$ by
$$
(x_{1}^{i_{1}}x_{2}^{i_{2}}\partial_{s})\diamond(x_{1}^{j_{1}}x_{2}^{j_{2}}\partial_{t})
:=\delta_{s,1}x_{1}^{i_{1}+j_{1}+1}x_{2}^{i_{2}+j_{2}}\partial_{t}
+\delta_{s,2}x_{1}^{i_{1}+j_{1}}x_{2}^{i_{2}+j_{2}+1}\partial_{t},
$$
for any $i_{1}, i_{2}, j_{1}, j_{2}\in\bz$ and $s, t\in\{1, 2\}$.
Then $(B, \diamond)$ is a $\bz$-graded perm algebra with the linear decomposition
$B=\oplus_{i\in\bz}A_{i}$, where
$$
B_{i}=\Big\{\sum_{k=1}^{2}f_{k}\partial_{k}\mid f_{k}\text{ is a homogeneous
polynomial with }\deg(f_{k})=i-1,\; k=1,2\Big\},
$$
for all $i\in\bz$.
\end{ex}

By Proposition \ref{pro:Lie+perm}, we have the following proposition.

\begin{pro}\label{pro:aff-Lie}
Let $(\g, [-,-])$ be a finite-dimensional Lie algebra and $(B=\oplus_{i\in\bz}B_{i},
\diamond)$ be a $\bz$-graded perm algebra. Define a binary operation on $\g\otimes B$ by
$$
(g_{1}, b_{1})\ast(g_{2}, b_{2}):=[g_{1}, g_{2}]\otimes(b_{1}\diamond b_{2}),
$$
for any $g_{1}, g_{2}\in\g$ and $b_{1}, b_{2}\in B$. Then $(\g\otimes B, \ast)$ is a
$\bz$-graded Leibniz algebra, which is called the {\bf induced $\bz$-graded Leibniz
algebra from $(\g, [-,-])$ by $(B, \diamond)$}.
\end{pro}

\begin{proof}
By Proposition \ref{pro:Lie+perm}, $(\g\otimes B, \ast)$ is a Leibniz algebra.
Since $(B=\oplus_{i\in\bz}B_{i}, \diamond)$ is $\bz$-graded, $(\g\otimes B, \ast)$
is a $\bz$-graded Leibniz algebra.
\end{proof}

To construct infinite-dimensional Leibniz coalgebras, we need to extend the codomain of the
comultiplication $\delta$ to allow infinite sums. Let $U=\oplus_{i\in\bz}U_{i}$ and
$V=\oplus_{j\in\bz}V_{j}$ be $\bz$-graded vector spaces.
We call the {\bf completed tensor product} of $U$ and $V$ to be the vector space
$$
U\,\hat{\otimes}\,V:=\prod_{i,j\in\bz}U_{i}\otimes V_{j}.
$$
If $U$ and $V$ are finite-dimensional, then $U\,\hat{\otimes}\,V$ is just the usual
tensor product $U\otimes V$. In general, an element in $U\,\hat{\otimes}\,V$ is an
infinite formal sum $\sum_{i,j\in\bz}X_{ij} $ with $X_{ij}\in U_{i}\otimes V_{j}$.
So $X_{ij}=\sum_{\alpha} u_{i, \alpha}\otimes v_{j, \alpha}$ for pure tensors
$u_{i, \alpha}\otimes v_{j, \alpha}\in U_{i}\otimes V_{j}$ with $\alpha$ in a finite
index set. Thus a general term of $U\,\hat{\otimes}\,V$ is a possibly infinite sum
$\sum_{i,j,\alpha}u_{i\alpha}\otimes v_{j\alpha}$, where $i, j\in\bz$ and $\alpha$ is
in a finite index set (which might depend on $i, j$). With these notations, for linear
maps $f: U\rightarrow U'$ and $g: V\rightarrow V'$, define
$$
f\,\hat{\otimes}\,g: U\,\hat{\otimes}\,V\rightarrow U'\,\hat{\otimes}\,V',
\qquad \sum_{i,j,\alpha}u_{i,\alpha}\otimes v_{j, \alpha}\mapsto
\sum_{i,j,\alpha} f(u_{i, \alpha})\otimes g(v_{j, \alpha}).
$$
The twist map $\tau$ has its completion
$\hat{\tau}: V\,\hat{\otimes}\,V\rightarrow V\,\hat{\otimes}\,V$,
$\sum_{i,j,\alpha}u_{i, \alpha}\otimes v_{j, \alpha}\mapsto
\sum_{i,j,\alpha}v_{j, \alpha}\otimes u_{i, \alpha}$.
Moreover, we define a (completed) comultiplication to be a linear map
$\delta: V\rightarrow V\,\hat{\otimes}\,V$,
$\delta(v):=\sum_{i, j, \alpha}v_{1, i, \alpha}\otimes v_{2, j, \alpha}$.
Then we have the well-defined map
$$
(\delta\,\hat{\otimes}\,\id)\delta(v)=(\delta\,\hat{\otimes}\,\id)
\Big(\sum_{i,j,\alpha}v_{1, i, \alpha}\otimes v_{2, j, \alpha}\Big)
:=\sum_{i,j,\alpha}\delta(v_{1, i, \alpha})\otimes v_{2, j, \alpha}
\in V\,\hat{\otimes}\,V\,\hat{\otimes}\,V.
$$

\begin{defi}\label{def:Pd-coa}
\begin{enumerate}
\item[$(i)$] A {\bf completed perm coalgebra} is a pair $(B, \nu)$, where
    $B=\oplus_{i\in\bz}B_{i}$ is a $\bz$-graded vector space and
    $\nu: B\rightarrow B\,\hat{\otimes}\,B$ is a linear map satisfying
    $$
    (\nu\,\hat{\otimes}\,\id)\nu=(\id\,\hat{\otimes}\,\nu)\nu
    =(\tau\,\hat{\otimes}\,\id)(\nu\,\hat{\otimes}\,\id)\nu.
    $$
\item[$(ii)$] A {\bf completed Leibniz coalgebra} is a pair $(A, \vartheta)$, where
    $A=\oplus_{i\in\bz}A_{i}$ is a $\bz$-graded vector space and
    $\vartheta: A\rightarrow A\,\hat{\otimes}\,A$ is a linear map satisfying
    $$
(\id\,\hat{\otimes}\,\vartheta)\vartheta=(\vartheta\,\hat{\otimes}\,\id)\vartheta
    +(\hat{\tau}\,\hat{\otimes}\,\id)(\id\,\hat{\otimes}\,\vartheta)\vartheta.
    $$
\end{enumerate}
\end{defi}

\begin{ex}[\cite{LZB}]\label{ex:grpermco}
Consider the $\bz$-graded vector space $B=\{f_{1}\partial_{1}+f_{2}\partial_{2}\mid
f_{1}, f_{2}\in\Bbbk[x_{1}^{\pm}, x_{2}^{\pm}]\}=\oplus_{i\in\bz}B_{i}$ given in
Example \ref{ex:grperm}. Define a linear map $\nu: B\rightarrow B\,\hat{\otimes}\,B$ by
\begin{align*}
\nu(x_{1}^{m}x_{2}^{n}\partial_{1})&=\sum_{i_{1}, i_{2}\in\bz}
\Big(x_{1}^{i_{1}}x_{2}^{i_{2}}\partial_{1}\otimes x_{1}^{m-i_{1}}
x_{2}^{n-i_{2}+1}\partial_{1} -x_{1}^{i_{1}}x_{2}^{i_{2}}\partial_{2}
\otimes x_{1}^{m-i_{1}+1}x_{2}^{n-i_{2}}\partial_{1}\Big), \\[-2mm]
\nu(x_{1}^{m}x_{2}^{n}\partial_{2})&=\sum_{i_{1}, i_{2}\in\bz}
\Big(x_{1}^{i_{1}}x_{2}^{i_{2}}\partial_{1}\otimes x_{1}^{m-i_{1}}
x_{2}^{n-i_{2}+1}\partial_{2}-x_{1}^{i_{1}}x_{2}^{i_{2}}\partial_{2}
\otimes x_{1}^{m-i_{1}+1}x_{2}^{n-i_{2}}\partial_{2}\Big),
\end{align*}
for any $m, n \in\bz$. Then $(B=\oplus_{i\in\bz}B_{i}, \nu)$ is a completed perm coalgebra.
\end{ex}

Similar to the proof of Proposition \ref{pro:perm-coLie}, we give the dual version
of Proposition \ref{pro:aff-Lie}.

\begin{pro}\label{pro:coperm-coLie}
Let $(\g, \delta)$ be a finite-dimensional Lie coalgebra and $(B=\oplus_{i\in\bz}B_{i}, \nu)$
be a completed perm coalgebra. Define a linear map $\vartheta: \g\otimes B\rightarrow
(\g\otimes B)\,\hat{\otimes}\,(\g\otimes B)$ by
$$
\vartheta(g\otimes b)=\delta(g)\bullet\nu(b)=\sum_{(g)}\sum_{i,j,\alpha}
(g_{(1)}\otimes b_{1,i,\alpha})\otimes(g_{(2)}\otimes b_{2,j,\alpha}),
$$
for any $g\in\g$ and $b\in B$, where $\delta(g)=\sum_{(p)}p_{(1)}\otimes p_{(2)}$,
in the Sweedler notation and $\nu(b)=\sum_{i,j,\alpha}b_{1,i,\alpha}\otimes b_{2,j,\alpha}$.
Then $(\g\otimes B, \vartheta)$ is a completed Leibniz coalgebra.
\end{pro}

Next, we consider the completed Leibniz bialgebra structure on the tensor
product of a Lie bialgebra and a $\bz$-graded perm algebra.

\begin{defi}\label{def:quad}
Let $\varpi(-, -)$ be a bilinear form on a $\bz$-graded perm algebra
$(B=\oplus_{i\in\bz}B_{i}, \diamond)$.

$(i)$ $\varpi(-, -)$ is called {\bf invariant}, if $\varpi(b_{1}\diamond b_{2},\; b_{3})=
        \varpi(b_{1},\; b_{2}\diamond b_{3}-b_{3}\diamond b_{2})$ for any
        $b_{1}, b_{2}, b_{3}\in B$;
$(ii)$ $\varpi(-, -)$ is called {\bf graded}, if there exists some $m\in\bz$
        such that $\varpi(B_{i}, B_{j})=0$ when $i+j+m\neq0$.\\
A {\bf quadratic $\bz$-graded perm algebra}, denoted by $(B=\oplus_{i\in\bz}B_{i},
\diamond, \varpi)$, is a $\bz$-graded perm algebra together with a skew-symmetric
invariant nondegenerate graded bilinear form.
In particular, if $B=B_{0}$, it is just the quadratic perm algebra.
\end{defi}

\begin{ex}[\cite{LZB}]\label{ex:qu-perm}
Let $(B=\oplus_{i\in\bz}B_{i}, \diamond)$ be the $\bz$-graded perm algebra given in
Example \ref{ex:grperm}, where $B=\{f_{1}\partial_{1}+f_{2}\partial_{2}\mid f_{1},
f_{2}\in\Bbbk[x_{1}^{\pm}, x_{2}^{\pm}]\}=\oplus_{i\in\bz}B_{i}$. Define a skew-symmetric
bilinear form $\varpi(-,-)$ on $(B=\oplus_{i\in\bz}B_{i}, \diamond)$ by
\begin{align*}
&\varpi(x_{1}^{i_{1}}x_{2}^{i_{2}}\partial_{2},\ \ x_{1}^{j_{1}}x_{2}^{j_{2}}\partial_{1})
=-\varpi(x_{1}^{j_{1}}x_{2}^{j_{2}}\partial_{1},\ \ x_{1}^{i_{1}}x_{2}^{i_{2}}\partial_{2})
=\delta_{i_{1}+j_{1}, 0}\delta_{i_{2}+j_{2}, 0}, \\
&\qquad\quad\varpi(x_{1}^{i_{1}}x_{2}^{i_{2}}\partial_{1},\ \
x_{1}^{j_{1}}x_{2}^{j_{2}}\partial_{1})=\varpi(x_{1}^{i_{1}}x_{2}^{i_{2}}\partial_{2},\ \
x_{1}^{j_{1}}x_{2}^{j_{2}}\partial_{2})=0,
\end{align*}
for any $i_{1}, i_{2}, j_{1}, j_{2}\in\bz$.
Then $(B=\oplus_{i\in\bz}B_{i}, \diamond, \varpi)$ is a quadratic $\bz$-graded perm algebra.
Moreover, $\{x_{1}^{-i_{1}}x_{2}^{-i_{2}}\partial_{2}$,\; $-x_{1}^{-i_{1}}
x_{2}^{-i_{2}}\partial_{1}\mid i_{1}, i_{2}\in\bz\}$ is the dual basis of
$\{x_{1}^{i_{1}}x_{2}^{i_{2}}\partial_{1},\; x_{1}^{i_{1}}x_{2}^{i_{2}}\partial_{2}\mid
i_{1}, i_{2}\in\bz\}$ with respect to $\varpi(-,-)$, consisting of homogeneous elements.
\end{ex}

For a quadratic $\bz$-graded perm algebra $(B=\oplus_{i\in\bz}B_{i}, \diamond, \varpi)$,
we have $\varpi(b_{1}\diamond b_{2},\; b_{3})=\varpi(b_{2},\; b_{1}\diamond b_{3})$
for any $b_{1}, b_{2}, b_{3}\in B$. Moreover, the skew-symmetric nondegenerate
bilinear form $\varpi(-,-)$ induces bilinear forms
$$
(\underbrace{B\,\hat{\otimes}\,B\,\hat{\otimes}\,\cdots\,\hat{\otimes}\,
B}_{n\text{-fold}})\otimes(\underbrace{B\otimes B\otimes\cdots
\otimes B}_{n\text{-fold}})\longrightarrow\Bbbk,
$$
for all $n\geq2$, which are denoted by $\hat{\varpi}(-,-)$, are defined by
$$
\hat{\varpi}\Big(\sum_{i_{1},\cdots,i_{n},\alpha} x_{1, i_{1}, \alpha}
\otimes\cdots\otimes x_{n, i_{n}, \alpha},\ \ y_{1}\otimes\cdots\otimes y_{n}\Big)
=\sum_{i_{1},\cdots,i_{n},\alpha}\prod_{j=1}^{n}
\varpi(x_{j, i_{j}, \alpha},\; y_{j}).
$$
Then, one can check that $\hat{\varpi}(-,-)$ is {\bf left nondegenerate}, i.e., if
$$
\hat{\varpi}\Big(\sum_{i_{1}, \cdots, i_{n},\alpha}x_{1, i_{1}, \alpha}
\otimes\cdots\otimes x_{n, i_{n}, \alpha},\ \ y_{1}\otimes\cdots\otimes y_{n}\Big)
=\hat{\varpi}\Big(\sum_{i_{1},\cdots,i_{n},\alpha} z_{1, i_{1}, \alpha}
\otimes\cdots\otimes z_{n, i_{n}, \alpha},\ \ y_{1}\otimes\cdots\otimes y_{n}\Big),
$$
for all homogeneous elements $y_{1}, y_{2},\cdots, y_{n}\in B$, then
$$
\sum_{i_{1},\cdots,i_{n},\alpha} x_{1, i_{1}, \alpha}
\otimes\cdots\otimes x_{n, i_{n}, \alpha}
=\sum_{i_{1},\cdots,i_{n},\alpha} z_{1, i_{1}, \alpha}
\otimes\cdots\otimes z_{n, i_{n}, \alpha}.
$$

By direct calculation, we have

\begin{lem}\label{lem:comp-dual}
Let $(B=\oplus_{i\in\bz}B_{i}, \diamond, \varpi)$ be a quadratic $\bz$-graded perm algebra.
Define a linear map $\nu_{\varpi}: B\rightarrow B\otimes B$ by
$\hat{\varpi}(\nu_{\varpi}(b_{1}),\; b_{2}\otimes b_{3})
=-\varpi(b_{1},\; b_{2}\diamond b_{3})$, for any $b_{1}, b_{2}, b_{3}\in B$.
Then $(B, \nu_{\varpi})$ is a completed perm coalgebra.
\end{lem}

\begin{ex}[\cite{LZB}]\label{ex:ind-coperm}
Consider the quadratic $\bz$-graded perm algebra $(B=\oplus_{i\in\bz}B_{i}, \cdot, \varpi)$
given in Example \ref{ex:qu-perm}. Then the induced completed perm coalgebra
$(B=\oplus_{i\in\bz}B_{i}, \nu_{\varpi})$ is just the completed perm coalgebra
$(B=\oplus_{i\in\bz}B_{i}, \nu)$ given in Example \ref{ex:grpermco}.
\end{ex}

Recall that a {\bf Lie bialgebra} is a triple $(\g, [-,-], \delta)$ such that
$(\g, [-,-])$ is a Lie algebra, $(\g, \delta)$ is a Lie coalgebra, and the following
compatibility condition holds:
$$
\delta([g_{1}, g_{2}])=(\ad_{\g}(g_{1})\otimes\id+\id\otimes\ad_{\g}(g_{1}))
(\delta(g_{2}))-(\ad_{\g}(g_{2})\otimes\id+\id\otimes\ad_{\g}(g_{2}))(\delta(g_{1})),
$$
where $\ad_{\g}(g_{1})(g_{2})=[g_{1}, g_{2}]$, for all $g_{1}, g_{2}\in\g$.

\begin{defi}\label{def:CASIbia}
A {\bf completed Leibniz bialgebra} is a triple $(A, \ast, \vartheta)$ consisting of a
$\bz$-graded vector space $A=\oplus_{i}A_{i}$ and linear maps $\ast: A\otimes
A\rightarrow A$ and $\vartheta: A\rightarrow A\,\hat{\otimes}\, A$ such that
\begin{enumerate}\itemsep=0pt
\item[$(i)$] $(A, \ast)$ is a $\bz$-graded Leibniz algebra;
\item[$(ii)$] $(A, \vartheta)$ is a completed Leibniz coalgebra;
\item[$(iii)$] for any $a_{1}, a_{2}\in A$,
\begin{align}
&\qquad\qquad\quad \hat{\tau}((\fr_{A}(a_{2})\,\hat{\otimes}\,\id)(\vartheta(a_{1})))
=(\fr_{A}(a_{1})\,\hat{\otimes}\,\id)(\vartheta(a_{2})),\label{CL1}\\
&\vartheta(a_{1}\ast a_{2})=(\id\,\hat{\otimes}\,\fr_{A}(a_{2})
-(\fl_{A}+\fr_{A})(a_{2})\,\hat{\otimes}\,\id)((\id+\hat{\tau})
(\vartheta(a_{1})))\label{CL2}\\[-1mm]
&\qquad\qquad\qquad\qquad\qquad\qquad\qquad +(\id\,\hat{\otimes}\,\fl_{A}(a_{1})
+\fl_{A}(a_{1})\,\hat{\otimes}\,\id)(\vartheta(a_{2})).\nonumber
\end{align}
\end{enumerate}
\end{defi}

Completed Lie bialgebras were introduced in \cite{HBG}. The completed Leibniz bialgebra
is a generalized completed Lie bialgebra.

\begin{thm}\label{thm:Lie+perm=L}
Let $(\g, [-,-], \delta)$ be a finite-dimensional Lie bialgebra, $(B=\oplus_{i\in\bz}B_{i},
\diamond, \varpi)$ be a quadratic $\bz$-graded perm algebra and $(\g\otimes B, \ast)$ be
the induced $\bz$-graded Leibniz algebra from $(\g, [-,-])$ by $(B, \diamond)$. Define
a linear map $\vartheta: \g\otimes B\rightarrow(\g\otimes B)\otimes(\g\otimes B)$ by
$$
\vartheta(g\otimes b)=\delta(g)\bullet\nu_{\varpi}(b)
:=\sum_{(g)}\sum_{i,j,\alpha}(g_{(1)}\otimes b_{1,i,\alpha})
\otimes(g_{(2)}\otimes b_{2,j,\alpha}),
$$
for any $g\in\g$ and $b\in B$, where $\delta(g)=\sum_{(g)}g_{(1)}
\otimes g_{(2)}$ in the Sweedler notation and $\nu_{\varpi}(b)=\sum_{i,j,\alpha}
b_{1,i,\alpha}\otimes b_{2,j,\alpha}$. Then $(\g\otimes B, \ast, \vartheta)$ is a
completed Leibniz bialgebra, which is called the {\bf completed Leibniz bialgebra induced
from $(\g, [-,-], \delta)$ by $(B, \diamond, \varpi)$}.
\end{thm}

\begin{proof}
First, by Proposition \ref{pro:coperm-coLie} and Lemma \ref{lem:comp-dual}, we get
that $(\g\otimes B, \vartheta)$ is a completed Leibniz coalgebra. Second,
for any $e, f\in B$, since
\begin{align*}
\hat{\varpi}\Big(\sum_{i,j,\alpha}b_{2,j,\alpha}\otimes(b_{1,i,\alpha}\diamond b'),\;
e\otimes f\Big)&=-\varpi(b,\;(b'\diamond f)\diamond e-(f\diamond b')\diamond e)=0,\\[-2mm]
\hat{\varpi}\Big(\sum_{i,j,\alpha}(b'_{1,i,\alpha}\diamond b)\otimes b'_{2,j,\alpha},\;
e\otimes f\Big)&=-\varpi(b',\;(b\diamond e)\diamond f-(e\diamond b)\diamond f)=0,
\end{align*}
we get $\hat{\tau}((\fr_{\g\otimes B}(g'\otimes b')\,\hat{\otimes}\,\id)
(\vartheta(g\otimes b)))=0=(\fr_{\g\otimes B}(g\otimes b)\,\hat{\otimes}\,\id)
(\vartheta(g'\otimes b'))$. Similarly, since $(\g, [-,-], \delta)$ is a Lie bialgebra,
i.e., $\tau\delta=-\delta$ and $\delta([g, g'])=(\ad_{\g}(g)\otimes\id+\id\otimes\ad_{\g}(g))
(\delta(g'))-(\ad_{\g}(g')\otimes\id+\id\otimes\ad_{\g}(g'))(\delta(g))$,
for any $g, g'\in\g$ and $b, b'\in B$, we get
\begin{align*}
&\; \vartheta((g\otimes b)\ast(g'\otimes b'))
-(\id\,\hat{\otimes}\,\fr_{\g\otimes B}(g'\otimes b')
-(\fl_{\g\otimes B}+\fr_{\g\otimes B})(g'\otimes b')\,\hat{\otimes}\,\id)((\id+\hat{\tau})
(\vartheta(g\otimes b)))\\[-1mm]
&\qquad\qquad\qquad\qquad-(\id\,\hat{\otimes}\,\fl_{\g\otimes B}(g\otimes b)
+\fl_{\g\otimes B}(g\otimes b)\,\hat{\otimes}\,\id)(\vartheta(g'\otimes b'))\\
=&\;\Big(\delta([g, g'])-(\ad_{\g}(g)\otimes\id+\id\otimes\ad_{\g}(g))
(\delta(g'))+(\ad_{\g}(g')\otimes\id+\id\otimes\ad_{\g}(g'))(\delta(g))\Big)
\nu_{\varpi}(b\diamond b')\\
=&\; 0.
\end{align*}
Thus, $(\g\otimes B, \ast, \vartheta)$ is a completed Leibniz bialgebra.
\end{proof}

Recall that a Lie bialgebra $(\g, [-,-], \delta)$ is called {\bf coboundary}
if there exists an element $r\in\g\otimes\g$ such that $\delta=\delta_{r}$,
where
\begin{align}
\delta_{r}(g)=(\id\otimes\ad_{\g}(g)+\ad_{\g}(g)\otimes\id)(r), \label{lie-cobo}
\end{align}
for any $g\in\g$. Let $(\g, [-,-])$ be a Lie algebra. An element
$r=\sum_{i}x_{i}\otimes y_{i}\in\g\otimes\g$ is said to be {\bf $\ad_{\g}$-invariant}
if $(\id\otimes\ad_{\g}(g)+\ad_{\g}(g)\otimes\id)(r)=0$ for all $g\in\g$. The equation
$$
\mathbf{C}_{r}:=[r_{12}, r_{13}]+[r_{13}, r_{23}]+[r_{12}, r_{23}]=0
$$
is called the {\bf classical Yang-Baxter equation} (or $\CYBE$) in $(\g, [-,-])$,
where $[r_{12}, r_{13}]=\sum_{i,j}[x_{i}, x_{j}]\otimes y_{i}\otimes y_{j}$,
$[r_{13}, r_{23}]=\sum_{i,j}x_{i}\otimes x_{j}\otimes[y_{i}, y_{j}]$ and
$[r_{12}, r_{23}]=\sum_{i,j}x_{i}\otimes[y_{i}, x_{j}]\otimes y_{j}$.

\begin{pro}[\cite{RS,LS}]\label{pro:lie-bia}
Let $(\g, [-,-])$ be a Lie algebra, $r\in\g\otimes\g$ and $\delta_{r}:
\g\rightarrow\g\otimes\g$ be the linear map defined by Eq. \eqref{lie-cobo}.
\begin{enumerate}\itemsep=0pt
\item[$(i)$] If $r$ is a skew-symmetric solution of the $\CYBE$ in $(\g, [-,-])$, then
     $(\g, [-,-], \delta_{r})$ is a Lie bialgebra, which is called a {\bf triangular Lie
     bialgebra} associated with $r$.
\item[$(ii)$] If $r$ is a solution of the $\CYBE$ in $(\g, [-,-])$ and $r+\tau(r)$ is
     $\ad_{\g}$-invariant, then $(\g, [-,-], \delta_{r})$ is a Lie bialgebra, which is
     called a {\bf quasi-triangular Lie bialgebra} associated with $r$.
\end{enumerate}
\end{pro}

Let $(A=\oplus_{i\in\bz}A_{i}, \ast)$ be a $\bz$-graded Leibniz algebra.
Suppose that $r=\sum_{i,j,\alpha}x_{i\alpha}\otimes y_{j\alpha}\in A\,\hat{\otimes}\,A$.
We denote $r_{12}\ast r_{13}:=\sum_{i,j,k,l,\alpha,\beta}(x_{i,\alpha}\ast x_{k,\beta})
\otimes y_{j,\alpha}\otimes y_{l,\beta}$, $r_{12}\ast r_{23}:=\sum_{i,j,k,l,\alpha,\beta}
x_{i,\alpha}\otimes(y_{j,\alpha}\ast x_{k,\beta})\otimes y_{l,\beta}$, $r_{23}\ast r_{12}:=
\sum_{i,j,k,l,\alpha,\beta}x_{k,\alpha}\otimes(x_{i,\beta}\ast y_{l,\beta})\otimes
y_{j,\alpha}$, $r_{23}\ast r_{13}:=\sum_{i,j,k,l,\alpha,\beta}x_{k,\alpha}\otimes
x_{i,\beta}\otimes(y_{j,\alpha}\ast y_{l,\beta})$. If $r\in A\,\hat{\otimes}\,A$ satisfies
$\mathbf{L}_{r}=r_{12}\ast r_{13}-r_{12}\ast r_{23}-r_{23}\ast r_{12}+r_{23}\ast r_{13}=0$
in $A\,\hat{\otimes}\,A\,\hat{\otimes}\,A$, then $r$ is called a
{\bf completed solution} of the $\LYBE$ in $(A=\oplus_{i\in\bz}A_{i}, \ast)$.
If $A=A_{0}$ is finite-dimensional, a completed solution of the $\LYBE$ in $(A, \ast)$
is just a solution of the $\LYBE$ in the Leibniz algebra $(A=A_{0}, \ast)$. The same
argument of the proof for \cite[Corollary 2.4.1]{Bai}
extends to the completed case. We obtain the following proposition.

\begin{pro}\label{pro:Leib-tri}
Let $(A=\oplus_{i\in\bz}A_{i}, \ast)$ be a $\bz$-graded Leibniz algebra
and $r\in A\,\hat{\otimes}\,A$ is a completed solution of the $\LYBE$ in
$(A=\oplus_{i\in\bz}A_{i}, \ast)$. Define a bilinear map $\vartheta_{r}: A\rightarrow
A\,\hat{\otimes}\,A$ by
\begin{align}
\vartheta_{r}(a):=\big((\fl_{A}+\fr_{A})(a)\,\hat{\otimes}\,\id)
-\id\,\hat{\otimes}\,\fr_{A}(a)\big)(r),  \label{cL-cobo}
\end{align}
for any $a\in A$. If $r$ is symmetric, i.e., $r=\hat{\tau}(r)$, then
$(A, \ast, \vartheta_{r})$ is a completed Leibniz bialgebra, which is called a
{\bf triangular completed Leibniz bialgebra} associated with $r$.
\end{pro}

Let $(A=\oplus_{i\in\bz}A_{i}, \ast)$ be a $\bz$-graded Leibniz algebra.
A completed Leibniz bialgebra $(A, \ast, \vartheta)$ is called {\bf coboundary} if
there exists an element $r\in A\,\hat{\otimes}\,A$ such that $\vartheta=\vartheta_{r}$.
Moreover, we call $r\in A\,\hat{\otimes}\,A$ is {\bf Leib-invariant} if
$$
\big((\fl_{A}+\fr_{A})(a)\,\hat{\otimes}\,\id)-\id\,\hat{\otimes}\,\fr_{A}(a)\big)(r)=0,
$$
for all $a\in A$. Similar to the proofs of Proposition 2.8 and Corollary 2.9 in
\cite{BLST}, we have the following conclusion.

\begin{pro}\label{pro:Leib-qtri}
Let $(A=\oplus_{i\in\bz}A_{i}, \ast)$ be a $\bz$-graded Leibniz algebra
and $r\in A\,\hat{\otimes}\,A$ is a completed solution of the $\LYBE$ in
$(A=\oplus_{i\in\bz}A_{i}, \ast)$. Define a bilinear map $\vartheta_{r}: A\rightarrow
A\,\hat{\otimes}\,A$ by Eq. \ref{cL-cobo} for any $a\in A$. If $r-\hat{\tau}(r)$ is
Leib-invariant, then $(A, \ast, \vartheta_{r})$ is a completed Leibniz bialgebra,
which is called a {\bf quasi-triangular completed Leibniz bialgebra} associated with $r$.
\end{pro}

Obviously, triangular completed Leibniz bialgebra is a special type of quasi-triangular
completed Leibniz bialgebra.
Next, we consider the relation between the solutions of the $\CYBE$ in a Lie
algebra and the solutions of the $\LYBE$ in the induced Leibniz algebra.
Let $(B=\oplus_{i\in\bz}B_{i}, \diamond, \varpi)$ be a quadratic $\bz$-graded perm
algebra and $\{e_{j}\}_{j\in\Omega}$ be a basis of $B=\oplus_{i\in\bz}B_{i}$ consisting
of homogeneous elements. Since $\varpi(-,-)$ is graded, skew-symmetric and nondegenerate,
we get a homogeneous dual basis $\{f_{j}\}_{j\in\Omega}$ of $B=\oplus_{i\in\bz}B_{i}$,
which is called the dual basis of $\{e_{i}\}_{i\in\Omega}$ with respect to $\varpi(-,-)$,
by $\varpi(f_{i}, e_{j})=\delta_{ij}$, where $\delta_{ij}$ is the Kronecker delta.

\begin{pro}\label{pro:CYBE-LYBE}
Let $(\g, [-,-])$ be a Lie algebra, $(B=\oplus_{i\in\bz}B_{i}, \diamond,
\varpi)$ be a quadratic $\bz$-graded perm algebra and $(\g\otimes B, \ast)$ be the
induced $\bz$-graded Leibniz algebra. Suppose that $r=\sum_{i}x_{i}\otimes y_{i}\in
\g\otimes\g$ is a solution of the $\CYBE$ in $(\g, [-,-])$. Then
\begin{align}
\widehat{r}=\sum_{i}\sum_{j\in\Omega}(x_{i}\otimes e_{j})\otimes(y_{i}\otimes f_{j})
\in(\g\otimes B)\,\hat{\otimes}\,(\g\otimes B)  \label{r-indL}
\end{align}
is a completed solution of the $\LYBE$ in $(A\otimes B, \ast)$,
where $\{e_{j}\}_{j\in\Omega}$ is a homogeneous basis of $B=\oplus_{i\in\bz}B_{i}$
and $\{f_{j}\}_{j\in\Omega}$ is the homogeneous dual basis of $\{e_{j}\}_{j\in\Omega}$
with respect to $\varpi(-,-)$. Moreover, we have
\begin{enumerate}\itemsep=0pt
\item[$(i)$] $\widehat{r}$ is symmetric if $r$ is skew-symmetric;
\item[$(ii)$] $\widehat{r}-\hat{\tau}(\widehat{r})$ is Leib-invariant if $r+\tau(r)$
     is $\ad_{\g}$-invariant.
\end{enumerate}
\end{pro}

\begin{proof}
First, for any $p, q\in\Omega$, we have
\begin{align*}
&\;\widehat{r}_{12}\ast\widehat{r}_{13}-\widehat{r}_{12}\ast\widehat{r}_{23}
-\widehat{r}_{23}\ast\widehat{r}_{12}+\widehat{r}_{23}\ast\widehat{r}_{13}\\
=&\;\sum_{i,j}\sum_{p,q}\Big(\big([x_{i}, x_{j}]\otimes y_{i}\otimes y_{j}\big)
\bullet\big((e_{p}\diamond e_{q})\otimes f_{p}\otimes f_{q}\big)
-\big(x_{i}\otimes[y_{i}, x_{j}]\otimes y_{j}\big)
\bullet\big(e_{p}\otimes(f_{p}\diamond e_{q})\otimes f_{q}\big)\\[-4mm]
&\qquad\quad-\big(x_{i}\otimes[x_{j}, y_{i}]\otimes y_{j}\big)
\bullet\big(e_{p}\otimes(e_{q}\diamond f_{p})\otimes f_{q}\big)
+\big(x_{i}\otimes x_{j}\otimes[y_{j}, y_{i}]\big)\bullet
\big(e_{p}\otimes e_{q}\otimes(f_{q}\diamond f_{p})\big)\Big).
\end{align*}
Moreover, for given $e_{s}, e_{u}, e_{v}\in B$, $s, u, v\in\Omega$, we have
\begin{align*}
\hat{\varpi}\Big(\sum_{p,q}(e_{p}\diamond e_{q})\otimes f_{p}\otimes f_{q},\ \
e_{s}\otimes e_{u}\otimes e_{v}\Big)&=\varpi(e_{u}\diamond e_{v},\; e_{s}),\\[-2mm]
\hat{\varpi}\Big(\sum_{p,q}(e_{p}\otimes(f_{p}\diamond e_{q})\otimes f_{q},\ \
e_{s}\otimes e_{u}\otimes e_{v}\Big)
&=\varpi(e_{v}\diamond e_{u}-e_{u}\diamond e_{v},\; e_{s}),\\[-2mm]
\hat{\varpi}\Big(\sum_{p,q}e_{p}\otimes(e_{q}\diamond f_{p})\otimes f_{q},\ \
e_{s}\otimes e_{u}\otimes e_{v}\Big)&=\varpi(e_{v}\diamond e_{u},\; e_{s}),\\[-2mm]
\hat{\varpi}\Big(\sum_{p,q}e_{p}\otimes e_{q}\otimes(f_{q}\diamond f_{p}),\ \
e_{s}\otimes e_{u}\otimes e_{v}\Big)&=\varpi(e_{u}\diamond e_{v},\; e_{s}).
\end{align*}
Thus, if we denote $\Phi_{1}, \Phi_{2}\in B\,\hat{\otimes}\,B\,\hat{\otimes}\,B$ by
$\hat{\varpi}(\Phi_{1},\; e_{s}\otimes e_{u}\otimes e_{v})=\varpi(e_{u}\diamond e_{v},\;
e_{s})$ and $\hat{\varpi}(\Phi_{2},\; e_{s}\otimes e_{u}\otimes e_{v})=\varpi(e_{v}\diamond
e_{u},\; e_{s})$, then we get
\begin{align*}
&\;\widehat{r}_{12}\ast\widehat{r}_{13}-\widehat{r}_{12}\ast\widehat{r}_{23}
-\widehat{r}_{23}\ast\widehat{r}_{12}+\widehat{r}_{23}\ast\widehat{r}_{13}\\
=&\; \mathbf{C}_{r}\bullet\Phi_{1}-\Big(\sum_{i,j}x_{i}\otimes[y_{i}, x_{j}]\otimes y_{j}
+x_{i}\otimes[x_{j}, y_{i}]\otimes y_{j}\Big)\bullet\Phi_{2}.
\end{align*}
Hence, $\widehat{r}$ is a completed solution of the $\LYBE$ in $(A\otimes B, \ast)$
if $r$ is a skew-symmetric solution of the $\CYBE$ in $(\g, [-,-])$.
Second, for any $e_{s}, e_{t}\in B$, $s, t\in\Omega$, we have
$$
\hat{\varpi}\Big(\sum_{j\in\Omega}e_{j}\otimes f_{j},\; e_{s}\otimes e_{t}\Big)
=\varpi(e_{s}, e_{t})=-\varpi(e_{t}, e_{s})
=-\hat{\varpi}\Big(\sum_{j\in\Omega}f_{j}\otimes e_{j},\; e_{s}\otimes e_{t}\Big).
$$
That is $\sum_{j}e_{j}\otimes f_{j}=-\sum_{j}f_{j}\otimes e_{j}$. Thus,
we get $\widehat{r}$ is symmetric if $r$ is skew-symmetric.

Finally, for any $e_{s}, e_{t}\in B$, $s, t\in\Omega$, note that
$$
\hat{\varpi}\Big(\sum_{j\in\Omega}(b\diamond e_{j})\otimes f_{j},\ \ e_{s}\otimes e_{t}\Big)
=\varpi(b\diamond e_{t},\; e_{s})
=-\hat{\varpi}\Big(\sum_{j\in\Omega}(b\diamond f_{j})\otimes e_{j},\ \ e_{s}\otimes e_{t}\Big).
$$
We get $\sum_{j\in\Omega}(b\diamond e_{j})\otimes f_{j}=-\sum_{j\in\Omega}(b\diamond f_{j})
\otimes e_{j}$. Similarly, we also have $\sum_{j\in\Omega}(e_{j}\diamond b)\otimes f_{j}
=-\sum_{j\in\Omega}(f_{j}\diamond b)\otimes e_{j}$ and $\sum_{j\in\Omega}e_{j}\otimes
(f_{j}\diamond b)=-\sum_{j\in\Omega}f_{j}\otimes(e_{j}\diamond b)=\sum_{j\in\Omega}
\big((b\diamond e_{j})\otimes f_{j}-(e_{j}\diamond b)\otimes f_{j}\big)$.
Thus, for any $g\in\g$ and $b\in B$,
\begin{align*}
&\;\Big((\fl_{\g\otimes B}+\fr_{\g\otimes B})(g\otimes b)\,\hat{\otimes}\,\id)
-\id\,\hat{\otimes}\,\fr_{\g\otimes B}(g\otimes b)\Big)(\widehat{r}-\hat{\tau}(\widehat{r}))\\
=&\;\sum_{i}\sum_{j\in\Omega}\Big(
([g, x_{i}]\otimes y_{i})\bullet((b\diamond e_{j})\otimes f_{j})
+([x_{i}, g]\otimes y_{i})\bullet((e_{j}\diamond b)\otimes f_{j})\\[-3mm]
&\qquad\qquad-(x_{i}\otimes[y_{i}, g])\bullet(e_{j}\otimes(f_{j}\diamond b))
-([g, y_{i}]\otimes x_{i})\bullet((b\diamond f_{j})\otimes e_{j})\\[-1mm]
&\qquad\qquad-([y_{i}, g]\otimes x_{i})\bullet((f_{j}\diamond b)\otimes e_{j})
+(y_{i}\otimes[x_{i}, g])\bullet(f_{j}\otimes(e_{j}\diamond b))\Big)\\
=&\;\sum_{i}\sum_{j\in\Omega}\Big(\big([g, x_{i}]\otimes y_{i}-x_{i}\otimes[y_{i}, g]
+[g, y_{i}]\otimes x_{i}-y_{i}\otimes[x_{i}, g]\big)
\bullet((b\diamond e_{j})\otimes f_{j})\\[-4mm]
&\qquad\qquad +\big([x_{i}, g]\otimes y_{i}+x_{i}\otimes[y_{i}, g]+[y_{i}, g]\otimes x_{i}
+y_{i}\otimes[x_{i}, g]\big)\bullet((e_{j}\diamond b)\otimes f_{j})\Big)\\
=&\; 0,
\end{align*}
if $r+\tau(r)$ is $\ad_{\g}$-invariant, i.e., $\sum_{i}\big([g, x_{i}]\otimes y_{i}
+x_{i}\otimes[g, y_{i}]+y_{i}\otimes[g, x_{i}]+[g, y_{i}]\otimes x_{i}\big)=0$.
That is, $\widehat{r}-\hat{\tau}(\widehat{r})$ is Leib-invariant if $r+\tau(r)$
is $\ad_{\g}$-invariant. The proof is finished.
\end{proof}

Triangular completed Leibniz bialgebras (resp. triangular Lie bialgebras) are
closely related to completed solutions (resp. solutions) of the classical Yang-Baxter
equation in a $\bz$-graded Leibniz algebra (resp. in a Lie algebra). By Proposition
\ref{pro:CYBE-LYBE}, we get the following conclusion.

\begin{thm}\label{thm:indu-triLib}
Let $(\g, [-,-], \delta)$ be a Lie bialgebra, $(B=\oplus_{i\in\bz}B_{i}, \diamond,
\varpi)$ be a quadratic $\bz$-graded perm algebra and $(\g\otimes B, \ast, \vartheta)$ be
the induced completed Leibniz bialgebra from $(\g, [-,-], \delta)$ by
$(B, \diamond, \varpi)$. If $\delta=\delta_{r}$ for some $r\in\g\otimes\g$,
then $\vartheta=\vartheta_{\widehat{r}}$, where $\widehat{r}$ is given by Eq. \eqref{r-indL}.

In particular, we get that $(\g\otimes B, \ast, \vartheta)$ is a coboundary (resp.
quasi-triangular, triangular) completed Leibniz bialgebra if
$(\g, [-,-], \delta)$ is coboundary (resp. quasi-triangular, triangular).
\end{thm}

\begin{proof}
Let $r=\sum_{i}x_{i}\otimes y_{i}\in\g\otimes\g$ and $\delta=\delta_{r}$. We now show
that $\vartheta=\vartheta_{\widehat{r}}$. That is, $(\g\otimes B, \ast, \vartheta)
=(\g\otimes B, \ast, \vartheta_{\widehat{r}})$ as completed Leibniz bialgebras.
In fact, for any $g\in\g$ and $b\in B$, we have
$$
\vartheta(g\otimes b)=\sum_{i}\sum_{l,k,\alpha}\Big(
(x_{i}\otimes[g, y_{i}])\bullet(b_{1,l,\alpha}\otimes b_{2,k,\alpha})
+([g, x_{i}]\otimes y_{i})\bullet(b_{1,l,\alpha}\otimes b_{2,k,\alpha})\Big),
$$
where $\delta_{r}(g)=(\id\otimes\ad_{\g}(g)+\ad_{\g}(g)\otimes\id)(r)
=\sum_{i}\big(x_{i}\otimes[g, y_{i}]+[g, x_{i}]\otimes y_{i}\big)$ and
$\nu_{\varpi}(b)=\sum_{l,k,\alpha}b_{1,l,\alpha}\otimes b_{2,k,\alpha}$. On the other hand,
\begin{align*}
\vartheta_{\widehat{r}}(g\otimes b)&=\big((\fl_{\g\otimes B}+\fr_{\g\otimes B})
(g\otimes b)\,\hat{\otimes}\,\id)-\id\,\hat{\otimes}\,\fr_{\g\otimes B}
(g\otimes b)\big)(\widehat{r})\\
&=\sum_{i}\sum_{j\in\Omega}\Big(\big([g, x_{i}]\otimes y_{i}\big)\bullet\big((b\diamond
e_{j})\otimes f_{j}-(e_{j}\diamond b)\otimes f_{j}\big)-\big(x_{i}\otimes[y_{i}, g]\big)
\bullet\big(e_{j}\otimes(f_{j}\diamond b)\big)\Big),
\end{align*}
where $\widehat{r}=\sum_{i}\sum_{j\in\Omega}(x_{i}\otimes e_{j})\otimes(y_{i}\otimes f_{j})$,
$\{e_{j}\}_{j\in\Omega}$ is a homogeneous basis of $B=\oplus_{i\in\bz}B_{i}$
and $\{f_{j}\}_{j\in\Omega}$ is the homogeneous dual basis of $\{e_{j}\}_{j\in\Omega}$
with respect to $\varpi(-,-)$.

For two homogeneous basis elements $e_{s}, e_{t}\in B$, $s, t\in\Omega$, since
$$
\hat{\varpi}\Big(\sum_{l,k,\alpha}b_{1,l,\alpha}\otimes b_{2,k,\alpha},\;
e_{s}\otimes e_{t}\Big)=\varpi(b\diamond e_{t}-e_{t}\diamond b,\; e_{s})
=\hat{\varpi}\Big(\sum_{j\in\Omega}e_{j}\otimes(f_{j}\diamond b),\; e_{s}\otimes e_{t}\Big),
$$
we get $\sum_{l,k,\alpha}b_{1,l,\alpha}\otimes b_{2,k,\alpha}=\sum_{j}e_{j}\otimes
(f_{j}\diamond b)$. Similarly, we also have$\sum_{l,k,\alpha}b_{1,l,\alpha}\otimes
b_{2,k,\alpha}=\sum_{j}\big((b\diamond e_{j})\otimes f_{j}-(e_{j}\diamond b)\otimes
f_{j}\big)$. Thus, $\vartheta_{\widehat{r}}(g\otimes b)=\vartheta(g\otimes b)$.
Therefore, $(\g\otimes B, \ast, \vartheta)$ is coboundary if $(\g, [-,-], \delta)$
is coboundary. Moreover, by Proposition \ref{pro:CYBE-LYBE}, we get that
$(\g\otimes B, \ast, \vartheta)$ is quasi-triangular (resp. triangular) if
$(\g, [-,-], \delta)$ is quasi-triangular (resp. triangular).
\end{proof}

Let $(\g, [-,-], \delta)$ be a Lie bialgebra and $r\in\g\otimes\g$.
It is worth noting that this theorem also gives the following commutative diagram:
$$
\xymatrix@C=3cm@R=0.5cm{
\txt{$r$ \\ {\tiny a solution of the $\CYBE$ in $(\g, [-,-])$}\\
{\tiny such that $r+\tau(r)$ is $\ad_{\g}$-invariant}}
\ar[d]_{{\rm Pro.}~\ref{pro:CYBE-LYBE}}\ar[r]^{{\rm Pro.}~\ref{pro:lie-bia}} &
\txt{$(\g, [-,-], \delta_{r})$ \\ {\tiny a quasi-triangular Lie bialgebra}}
\ar[d]^{{\rm Thm.}~\ref{thm:indu-triLib}}_{{\rm Thm.}~\ref{thm:Lie+perm=L}} \\
\txt{$\widehat{r}$ \\ {\tiny a completed solution of the $\LYBE$}\\ {\tiny
in $(\g\otimes B, \ast)$ such that $\widehat{r}-\hat{\tau}(\widehat{r})$ is Leib-invariant}}
\ar[r]^{{\rm Pro.}~\ref{pro:Leib-qtri}\quad} &
\txt{$(\g\otimes B, \ast, \vartheta)=(\g\otimes B, \ast, \vartheta_{\widehat{r}})$ \\
{\tiny a completed quasi-triangular Leibniz bialgebra}}}
$$

\begin{rmk}\label{rmk:indu-facLib}
$(i)$ The factorizable Lie bialgebra received further research in \cite{LS}.
If we consider the finite-dimensional Leibniz bialgebra induced from Lie bialgebras,
similar to the proof of Theorem \ref{thm:indu-sLbia}, we can also obtain that
$(\g\otimes B, \ast, \vartheta)$ is factorizable if $(\g, [-,-], \delta)$ is factorizable.

$(ii)$ Similar to the method in this section, we can also construct infinite-dimensional
Leibniz bialgebras from finite-dimensional perm bialgebras.
\end{rmk}

At the end of this paper, we give a simple example for infinite-dimensional
Leibniz bialgebras.

\begin{ex}\label{ex:ind-Leibniz}
Let $(\g=\Bbbk\{e_{1}, e_{2}\}, [-,-])$ be a nontrivial Lie algebra, i.e.,
$[e_{1}, e_{2}]=e_{2}$. Then one can check that $r=e_{2}\otimes e_{1}-e_{1}\otimes e_{2}$
is a skew-symmetric solution of the $\CYBE$ in $(\g, [-,-])$. Therefore, we obtain a
Lie bialgebra $(\g, [-,-], \delta_{r})$, where the coproduct $\delta_{r}$ is given by
$\delta_{r}(e_{1})=e_{2}\otimes e_{1}-e_{1}\otimes e_{2}$ and $\delta_{r}(e_{2})=0$.
Consider the quadratic $\bz$-graded perm algebra $(B, \diamond, \varpi)$ given in
Example \ref{ex:qu-perm}. By Theorem \ref{thm:Lie+perm=L}, we get a completed
Leibniz bialgebra $(\g\otimes B, \ast, \vartheta)$, where
\begin{enumerate}\itemsep=0pt
\item[$(i)$] as vector spaces $\g\otimes B=\bigoplus_{i}(\g\otimes B)_{i}$,
$(\g\otimes B)_{i}=\big\{\sum_{j=1}^{2}\sum_{k=1}^{2}f_{k,j}\partial_{k}e_{j}\mid f_{k,j}$
is a homogeneous polynomial with $\deg(f_{k,j})=i-1\big\}$;
\item[$(ii)$] the product $\ast$ is given by
\begin{align*}
(x_{1}^{i_{1}}x_{2}^{i_{2}}\partial_{s}e_{1})\diamond
(x_{1}^{j_{1}}x_{2}^{j_{2}}\partial_{t}e_{2})
&=\delta_{s,1}x_{1}^{i_{1}+j_{1}+1}x_{2}^{i_{2}+j_{2}}\partial_{t}e_{2}
+\delta_{s,2}x_{1}^{i_{1}+j_{1}}x_{2}^{i_{2}+j_{2}+1}\partial_{t}e_{2},\\
(x_{1}^{i_{1}}x_{2}^{i_{2}}\partial_{s}e_{2})\diamond
(x_{1}^{j_{1}}x_{2}^{j_{2}}\partial_{t}e_{1})
&=-\delta_{s,1}x_{1}^{i_{1}+j_{1}+1}x_{2}^{i_{2}+j_{2}}\partial_{t}e_{2}
-\delta_{s,2}x_{1}^{i_{1}+j_{1}}x_{2}^{i_{2}+j_{2}+1}\partial_{t}e_{2},
\end{align*}
and others are all zero;
\item[$(iii)$] the coproduct $\vartheta$ is given by
\begin{align*}
\vartheta(x_{1}^{m}x_{2}^{n}\partial_{1}e_{1})&=\sum_{i_{1}, i_{2}\in\bz}
\Big(x_{1}^{i_{1}}x_{2}^{i_{2}}\partial_{1}e_{2}\otimes x_{1}^{m-i_{1}}
x_{2}^{n-i_{2}+1}\partial_{1}e_{1}-x_{1}^{i_{1}}x_{2}^{i_{2}}\partial_{2}e_{2}
\otimes x_{1}^{m-i_{1}+1}x_{2}^{n-i_{2}}\partial_{1}e_{1}\\[-5mm]
&\qquad\qquad -x_{1}^{i_{1}}x_{2}^{i_{2}}\partial_{1}e_{1}\otimes x_{1}^{m-i_{1}}
x_{2}^{n-i_{2}+1}\partial_{1}e_{2}+x_{1}^{i_{1}}x_{2}^{i_{2}}\partial_{2}e_{1}
\otimes x_{1}^{m-i_{1}+1}x_{2}^{n-i_{2}}\partial_{1}e_{2}\Big), \\[-1mm]
\vartheta(x_{1}^{m}x_{2}^{n}\partial_{2}e_{1})&=\sum_{i_{1}, i_{2}\in\bz}
\Big(x_{1}^{i_{1}}x_{2}^{i_{2}}\partial_{1}e_{2}\otimes x_{1}^{m-i_{1}}
x_{2}^{n-i_{2}+1}\partial_{2}e_{1}-x_{1}^{i_{1}}x_{2}^{i_{2}}\partial_{2}e_{2}
\otimes x_{1}^{m-i_{1}+1}x_{2}^{n-i_{2}}\partial_{2}e_{1}\\[-5mm]
&\qquad\qquad -x_{1}^{i_{1}}x_{2}^{i_{2}}\partial_{1}e_{1}\otimes x_{1}^{m-i_{1}}
x_{2}^{n-i_{2}+1}\partial_{2}e_{2}+x_{1}^{i_{1}}x_{2}^{i_{2}}\partial_{2}e_{1}
\otimes x_{1}^{m-i_{1}+1}x_{2}^{n-i_{2}}\partial_{2}e_{2}\Big),
\end{align*}
and others are all zero.
\end{enumerate}

On the other hand, note that $\{x_{1}^{i_{1}}x_{2}^{i_{2}}\partial_{1},\;
x_{1}^{i_{1}}x_{2}^{i_{2}}\partial_{2}\mid i_{1}, i_{2}\in\bz\}$ is a homogeneous
basis of $B$ and its dual basis with respect to $\varpi(-,-)$ is given by
$\{x_{1}^{-i_{1}}x_{2}^{-i_{2}}\partial_{2},\; -x_{1}^{-i_{1}}x_{2}^{-i_{2}}\partial_{1}
\mid i_{1}, i_{2}\in\bz\}$. Let
\begin{align*}
\widehat{r}:&=\sum_{i_{1}, i_{2}}\Big(
x_{1}^{i_{1}}x_{2}^{i_{2}}\partial_{1}e_{2}\otimes
x_{1}^{-i_{1}}x_{2}^{-i_{2}}\partial_{2}e_{1}
-x_{1}^{i_{1}}x_{2}^{i_{2}}\partial_{2}e_{2}\otimes
x_{1}^{-i_{1}}x_{2}^{-i_{2}}\partial_{1}e_{1}\\[-5mm]
&\qquad\qquad-x_{1}^{i_{1}}x_{2}^{i_{2}}\partial_{1}e_{1}\otimes
x_{1}^{-i_{1}}x_{2}^{-i_{2}}\partial_{2}e_{2}
+x_{1}^{i_{1}}x_{2}^{i_{2}}\partial_{2}e_{1}\otimes
x_{1}^{-i_{1}}x_{2}^{-i_{2}}\partial_{1}e_{2}\Big).
\end{align*}
Then one can check that $\widehat{r}$ is a symmetric solution of the $\LYBE$
in $(\g\otimes B, \ast)$ and the coproduct $\vartheta_{\widehat{r}}$ induced
by $\widehat{r}$ is just the coproduct $\vartheta$ given above.
\end{ex}

\bigskip
\noindent
{\bf Acknowledgements. } This work was financially supported by National
Natural Science Foundation of China (No. 11771122).

\smallskip
\noindent
{\bf Declaration of interests.} The authors have no conflicts of interest to disclose.

\smallskip
\noindent
{\bf Data availability.} Data sharing is not applicable to this article as no new data were
created or analyzed in this study.

 \end{document}